\documentclass[11pt,reqno,a4paper]{amsart}
\usepackage[utf8]{inputenc}
\usepackage[T1]{fontenc}
\usepackage[british]{babel}
\usepackage[a4paper,textwidth=160mm,textheight=239mm,left=25mm,top=29mm,includehead=false,includefoot=false]{geometry}
\usepackage{microtype}
\usepackage{mathpazo}
\usepackage{euler}

\usepackage{amsmath,amssymb,amsthm,amscd,mathtools,mathrsfs,needspace,booktabs,array}
\AtBeginDocument{\renewcommand{\hbar}{\hslash}}
\usepackage{xcolor}
\usepackage{hyperref}
\hypersetup{colorlinks=true,linkcolor=blue!60!black,citecolor=blue!60!black,urlcolor=blue!60!black,
 pdftitle={Hamiltonian unfoldings and quantum obstructions},
 pdfauthor={Mauricio Correa and Simone Noja},
 pdfsubject={Hamiltonian transport for prescribed star products}}
\AtBeginDocument{%
  \SetMathAlphabet{\mathsf}{normal}{OT1}{cmss}{m}{n}%
  \SetMathAlphabet{\mathsf}{bold}{OT1}{cmss}{bx}{n}%
}

\newtheorem{theorem}{Theorem}[section]
\newtheorem{proposition}[theorem]{Proposition}
\newtheorem{lemma}[theorem]{Lemma}
\newtheorem{corollary}[theorem]{Corollary}
\theoremstyle{definition}
\newtheorem{definition}[theorem]{Definition}
\newtheorem{example}[theorem]{Example}
\theoremstyle{remark}
\newtheorem{remark}[theorem]{Remark}
\numberwithin{equation}{section}

\DeclareMathOperator{\Der}{Der}
\DeclareMathOperator{\Hom}{Hom}
\DeclareMathOperator{\Sym}{Sym}
\DeclareMathOperator{\Spec}{Spec}
\DeclareMathOperator{\Alt}{Alt}
\DeclareMathOperator{\ad}{ad}
\DeclareMathOperator{\im}{im}
\newcommand{\cL}{\mathcal L}

\newcommand{\dd}{d_{\mathrm v}}

\newcommand{\gPproj}{\mathfrak P_\pi}

\newcommand{\id}{\mathrm{id}}

\title[Hamiltonian unfoldings and quantum obstructions]{Hamiltonian unfoldings and quantum obstructions}
\author[M. Corr\^ea]{Mauricio Corr\^ea}
\address[M. Corr\^ea]{Dipartimento di Matematica, Universit\`a degli Studi di Bari, Via E. Orabona 4, I-70125 Bari, Italy}
\email{mauricio.correa.mat@gmail.com, mauricio.barros@uniba.it}
\author[S. Noja]{Simone Noja}
\address[S. Noja]{Dipartimento di Matematica, Universit\`a degli Studi di Bari, Via E. Orabona 4, I-70125 Bari, Italy}
\email{simone.noja@uniba.it}
\date{}
\begin{document}

\begin{abstract}
We study when classical transport in a Poisson family survives a
prescribed deformation quantisation. For a reductive Lie algebra bundle over an affine 
complex variety of positive dimension, we prove that this transport lifts to every star
product if and only if the centre of the fibre Lie algebra has dimension
at most one. For general Poisson families we construct successive
obstruction classes and give examples in which the first obstruction
occurs at any prescribed order. Finally, for a Heisenberg family we
compare the prescribed product with a Kontsevich product and identify
the obstruction to lifting the prescribed curving together with the
connection on polynomial boundary observables of the Poisson sigma model.
\end{abstract}

\maketitle

\tableofcontents

\Needspace{8\baselineskip}
\section{Introduction}\label{sec:problem}

In Hamiltonian mechanics, observables form a commutative algebra
equipped with a Poisson bracket. When a system depends on external
parameters, comparing its observables at different parameter values
requires a rule for transport that respects their Poisson brackets.
Deformation quantisation replaces the commutative product by a
formal associative product, called a star product, whose first-order
commutator recovers the Poisson bracket~\cite{WaldmannIntro}.
Does a given rule for classical transport survive this change of
product?

\medskip

We work with smooth affine Poisson families in characteristic zero.
The base represents the space of external parameters, and the fibre
over a point represents the classical phase space at those parameter
values. A Poisson connection describes infinitesimal transport along
the base. In two commuting parameter directions its curvature measures
how taking the transports in opposite orders changes an observable.
For instance, the coordinate derivatives of two local parameters $s$
and $t$ commute on the base, while the corresponding derivatives
$H_s$ and $H_t$ on observables can have a nonzero commutator.

We consider connections for which this discrepancy is itself
Hamiltonian, and retain a choice of the Hamiltonians generating it.
In physical terms, the infinitesimal change around a small loop in
parameter space is generated by a fibre observable, just as time
evolution in Hamiltonian mechanics is generated by the energy.
For the two directions above, this means that there is an observable
$C_{st}$ such that, for every fibre observable $f$,
\[
                        [H_s,H_t]f=\{C_{st},f\}.
\]
Together the chosen generators form a two-form $C$ on the base
with values in fibre observables.

Adding a Casimir, an observable whose Poisson bracket with every
observable vanishes, leaves the motion generated by $C_{st}$
unchanged. The ordinary Bianchi identity implies that the covariant
exterior derivative $d_HC$ is a three-form whose coefficients
Poisson commute with every fibre observable. We require this
three-form to vanish. To see why, consider the bracket that
combines infinitesimal parameter variations and fibre observables.
In this bracket, $C_{st}$ is the observable produced by commuting
the two parameter variations. For three such variations the Jacobi
identity gives $d_HC=0$. The connection and its chosen Hamiltonians
then form a \emph{Hamiltonian unfolding}, in which base vector
fields and fibre observables combine into a single Lie bracket.

\Needspace{5\baselineskip}
We now fix a star product on the fibres and ask whether the classical
transport lifts to it. Transporting the product of two observables
should give the same result as transporting each observable and then
multiplying. Infinitesimally, this means that the quantum connection
$D$ must satisfy the Leibniz rule in every base direction $\xi$:
\[
       D_\xi(f\star g)=(D_\xi f)\star g+f\star(D_\xi g).
\]
We also ask that the curvature of $D$ be generated by quantum
Hamiltonians. These act on observables through the divided commutator
$[a,f]_\star/\hbar$, where $[a,f]_\star=a\star f-f\star a$.
As in the classical family, the chosen Hamiltonians must be
covariantly closed. A lift recovers both the original connection
and its chosen Hamiltonians at $\hbar=0$.

A simple example shows why the choice of star product matters even
for a constant family. Let $t$ be a single external parameter, and
take the polynomial algebra $\mathbb C[u,v]$ with zero Poisson
bracket on every fibre. Classically, we can vary $t$ while keeping
$u$ and $v$ fixed. Now prescribe the Weyl star product with $t$
central and
\[
                         [u,v]_\star=\hbar^2t.
\]
This product still quantises the zero Poisson bracket, since its
commutator starts at order $\hbar^2$. Suppose that differentiation
along $t$ lifted to a quantum derivation $D$. Since $D(t)=1$,
applying $D$ to the displayed relation would give
\[
                [D(u),v]_\star+[u,D(v)]_\star=\hbar^2.
\]
At $t=0$ the fibre algebra is commutative, so both commutators on
the left vanish, while the right-hand side remains $\hbar^2$. Thus
even with a single parameter and a fixed classical fibre, the
prescribed product can obstruct transport at the level of the
Leibniz rule.

\Needspace{5\baselineskip}
\medskip

The plane in the preceding example is the dual of a two-dimensional
abelian Lie algebra. Its zero Poisson bracket is the simplest
instance of the linear Poisson bracket on the dual of a Lie algebra.
In general, linear observables on the dual inherit the Lie bracket,
and the Leibniz rule extends it to all polynomial observables.
Our main result concerns families of this kind. Let $\mathfrak g$
be a reductive Lie algebra bundle over a smooth connected affine
complex variety of positive dimension, and assume that the bundle
becomes constant after an \'etale base change. For any Hamiltonian
unfolding on $\mathfrak g^*$, we prove that it lifts to every
relative differential star product if and only if the centre of the
fibre Lie algebra has dimension at most one
(Theorem~\ref{thm:threshold}).

To see why the centre enters, recall that a reductive Lie algebra
splits as the direct sum of its semisimple part and its centre.
We first lift the semisimple generators so that their divided
commutators reproduce their Lie brackets. The remaining lifting
equations then lie in their centraliser, the algebra of quantum
observables commuting with all the lifted generators. When the
centre has dimension at most one, every element of this centraliser
can be written locally, to any finite order in $\hbar$, as a
polynomial in at most one generator with central coefficients.
The centraliser is therefore commutative, and we use this
description to lift the remaining transport and curvature data.

With two or more central directions, the choice of product becomes
decisive. The Poincar\'e--Birkhoff--Witt product, obtained from the
enveloping algebra by symmetrisation, admits a lift for every
unfolding, whatever the dimension of the centre. Other products
allow the commutator of two classically central observables to vary
along the base, as in the preceding example. For any prescribed
order $m\geq 1$, we construct such a product for which every
unfolding lifts through order $m-1$ and none lifts through order $m$.

For a general Poisson family, we return to a single prescribed star
product and ask how far a given Hamiltonian unfolding lifts. At first
order, the problem is to correct the classical connection and its
curvature Hamiltonians so that they satisfy the quantum transport
equations. We construct a cohomology class whose vanishing is
equivalent to the existence of this first correction
(Theorem~\ref{thm:main}). At each subsequent order, the corrections
already chosen determine a class in the same cohomology group.
Its vanishing tells us whether that particular partial lift extends
by one more order (Theorem~\ref{thm:tower}).

The curvature Hamiltonians can themselves account for the
obstruction. Suppose that we have lifted the connection and found
quantum Hamiltonians that generate its curvature and recover the
chosen classical ones. Their covariant exterior derivative then
takes values in the centre of the quantum algebra. The resulting
degree-three cohomology class vanishes if and only if, with the
connection held fixed, the Hamiltonians can be chosen covariantly
closed (Theorem~\ref{thm:central-bianchi}).

We exhibit this phenomenon on a family whose fibres are duals of
the Heisenberg Lie algebra. In fibre coordinates $x,y,z$, the
bracket is determined by $\{x,y\}=z$, with $z$ a Casimir. We
choose a flat connection that already preserves the prescribed
star product. Because $z$ is a classical Casimir, any polynomial
$f(z)$ generates zero classical motion, and we can take
$f(z)\omega$ as the curvature Hamiltonians, where $\omega$ is a
fixed closed two-form on the parameter torus. The choice of $f$
matters after quantisation. On the whole family, even allowing
corrections to the connection, a first-order lift exists if and
only if $f$ is constant. On the regular locus $z\ne 0$, we
construct a corrected connection that gives a lift to all orders
for every $f$.

Finally, we compare the star product in this example with the
Kontsevich product for a formal deformation of the Heisenberg
Poisson bracket. We construct an equivalence that intertwines
the chosen connection and preserves the quantum Hamiltonians
used in the calculation. In the perturbative disk model of Cattaneo and
Felder~\cite{CF}, the given connection already acts by derivations on
polynomial boundary observables. For nonconstant $f$, the obstruction
prevents the choice of covariantly closed curvature Hamiltonians
lifting $C_f$, even when quantum corrections to the connection are
allowed. Here the observables depend algebraically on the external
parameters.

\Needspace{9\baselineskip}
\medskip
\noindent
The paper is structured as follows. Section~\ref{sec:classical}
formulates the lifting problem for a prescribed star product.
Section~\ref{sec:reductive} proves the criterion for duals of
reductive Lie algebra bundles by reducing the lifting equations to
the centraliser of the semisimple generators. In Section~\ref{sec:quantum}
we construct the obstruction classes for a general Poisson family and
use them to decide whether a given partial lift extends by one order.
Section~\ref{sec:heis} computes the obstruction for the Heisenberg
family and shows that the same obstruction to lifting the pair
$(H,C_f)$ occurs for the Kontsevich product on polynomial boundary
observables.
Appendix~\ref{app:successive-obstructions} gives the proof of the
successive obstruction theorem.

\Needspace{7\baselineskip}

\section{Hamiltonian unfoldings and their lifts}\label{sec:classical}

This section gives an algebraic formulation of the transport problem
introduced above. A Hamiltonian unfolding consists of a Poisson
connection together with a chosen two-form of observables generating
its curvature. These data allow us to combine base vector fields
and fibre observables into a single Lie bracket, whose Jacobi
identity expresses their compatibility. We then formulate the
corresponding lifting problem for a prescribed star product,
requiring the quantum connection and its curvature Hamiltonians
to recover the chosen classical data.
\medskip

\noindent Subsection~\ref{sec:affine-setting} fixes the algebraic setting and
distinguishes differentiation along a fibre from differentiation in
a base direction. In Subsection~\ref{sec:classical-splittings} we
combine base vector fields and fibre observables in a single Lie
bracket. Choosing a splitting recovers the connection and its
curvature Hamiltonians, and changing the splitting explains when
two such descriptions are equivalent. Subsection~\ref{sec:lifts}
translates this bracket into equations for a lift to a prescribed
star product, both to finite order and as a formal series.
In Subsection~\ref{sec:central-bianchi} we hold the lifted connection
fixed. Once Hamiltonians generating its curvature have been found
with the prescribed classical limits, we ask whether a central
correction can make them covariantly closed.

\subsection{Algebraic setting and notation}\label{sec:affine-setting}
We use algebraic differential forms and operators throughout.
For deformation quantisation in the algebraic setting see~\cite{Kon}.
Formal series are completed in powers of $\hbar$. Continuity refers
to this topology,
in which a sequence converges when it eventually stabilises modulo
each power of $\hbar$.

Let $k$ be a field of characteristic zero, $R$ a smooth affine $k$-algebra,
and $A$ a smooth faithfully flat $R$-algebra of finite presentation.
Put $T=\Der_k R$. For any $R$-module $M$, write
$\Omega_T^p(M)=\Hom_R(\wedge_R^pT,M)$ for alternating base forms
with values in $M$. For $s\in\Spec R$, $\kappa(s)$ denotes its
residue field.
Elements of $\Der_R A$ differentiate along the fibres, and elements
of $T$ differentiate base parameters.
Smoothness makes $T$ finite projective, so $\Omega_T^p(-)$ preserves
exact sequences. Let $\pi\in\wedge_A^2\Der_R A$ be a relative
Poisson bivector, with bracket $\{-,-\}$.
The relative Poisson complex is
\[
 C_\pi^q(A/R)=\bigwedge_A^q\Der_R A,\qquad d_\pi=[\pi,-],
 \qquad H_\pi^q(A/R)=H^q(C_\pi^\bullet(A/R)).
\]
Here $[-,-]$ is the Schouten bracket, with $[Y,f]=Y(f)$. Thus
$[\pi,f]=-X_f$ for the Hamiltonian field $X_f(g)=\{f,g\}$.
Since $\pi$ is relative, $d_\pi$ is $R$-linear and each $H_\pi^q(A/R)$
is an $R$-module. The Poisson centre
$Z_\pi=\{f\in A:X_f=0\}$ consists of the \emph{Casimir functions}.
We call directions along the fibres vertical, and base directions
lifted by a connection horizontal.

\medskip
\noindent
Table~\ref{tab:algebraic-notation} separates the base, fibre and formal
directions. Objects introduced later are included for reference.
\par
\begin{table}[!htbp]
\caption{Notation for the algebraic setting}\label{tab:algebraic-notation}
\centering
\renewcommand{\arraystretch}{1.16}
\begin{tabular}{@{}p{.23\textwidth}p{.71\textwidth}@{}}
\toprule
Notation & Meaning \\
\midrule
$R$, $A$ & Base algebra and coordinate algebra of the total family;
the fibre over $s$ is $A\otimes_R\kappa(s)$. \\
$T$, $d_{\mathrm v}$ & $\Der_kR$ and the relative universal derivation
$A\to\Omega^1_{A/R}$. \\
$\Omega_T^p(M)$ & Base $p$-forms with coefficients in $M$;
$d_H$ is the covariant exterior derivative induced by $H$. \\
$H$, $C$ & Classical Poisson connection and its specified curving. \\
$A_\hbar$, $D$, $C_\hbar$ & Prescribed formal star algebra, quantum
derivation connection, and quantum curving. \\
$\mu(a,b)$, $I_a$ & Divided commutator $[a,b]_\star/\hbar$ and
its action $\mu(a,-)$. \\
$Z_\pi$, $Z_\hbar$ & Poisson centre of $A$ and associative centre
of $A_\hbar$. \\
$\overline Z_\hbar$ & $Z_\hbar/\hbar Z_\hbar$, naturally an
$R$-subalgebra of $Z_\pi$ consisting of Casimirs with central lifts. \\
$K_{A/R}$ & Relative constants $\ker d_{\mathrm v}$, a subalgebra of $Z_\pi$. \\
$E$ & Extension over $R$ of base vector fields by fibre observables
with their Poisson bracket. \\
\bottomrule
\end{tabular}
\end{table}
\medskip
\noindent
A \emph{$T$-connection} on an $R$-module $M$ is an $R$-linear
assignment $\xi\mapsto\nabla^M_\xi\in\operatorname{End}_k(M)$ with
$\nabla^M_\xi(rm)=\xi(r)m+r\nabla^M_\xi(m)$. It is \emph{flat} when
$[\nabla^M_\xi,\nabla^M_\eta]=\nabla^M_{[\xi,\eta]}$.
For such a flat connection, $H^p(T;M)$ denotes the cohomology of
$\Omega_T^\bullet(M)$ with its covariant exterior differential, and
$Z^p(T;M)$ denotes its cocycles. This is Lie--Rinehart cohomology
\cite[Section~4]{Rinehart}. The differential is $k$-linear and
differentiates base coefficients; the cohomology is therefore
naturally a $k$-vector space.

\subsection{Connections, curvings and changes of splitting}\label{sec:classical-splittings}
A connection lifts base vector fields to the total space. We describe
these lifts first, then identify their curvature in the extension bracket.

\Needspace{7\baselineskip}
The \emph{tangent sequence} of $\Spec A\to\Spec R$ is
\begin{equation}\label{eq:tangent-sequence}
0\longrightarrow\Der_R A\longrightarrow\Der_k A
 \xrightarrow{\operatorname{res}_R}A\otimes_R T\longrightarrow0.
\end{equation}
The map $\operatorname{res}_R$ restricts a derivation to $R$, using
$\Hom_R(\Omega^1_{R/k},A)\simeq A\otimes_R T$. Exactness follows by
dualising the split cotangent sequence for the smooth map $R\to A$
\cite[Tag~031K]{Stacks}. A horizontal splitting is an $A$-linear
right inverse to $\operatorname{res}_R$.

Following \cite[Section~2]{Rinehart}, we use
\emph{Lie--Rinehart algebra over $R$} to mean an $R$-module
$E$ with a $k$-Lie bracket and an $R$-linear Lie algebra map
$a:E\to T$, called the anchor, satisfying
$[e,re']=r[e,e']+a(e)(r)e'$.
When the module is finite projective, this is the description by
global sections of an algebraic \emph{Lie algebroid} on $\Spec R$.

\begin{definition}[Projectable Poisson fields and Poisson connections]
\label{def:poisson-connection}
A \emph{projectable Poisson vector field} is a pair $(\xi,Y)$ with $\xi\in T$,
$Y\in\Der_k A$, $Y|_R=\xi$, and $\cL_Y\pi=0$. These pairs form the
Lie--Rinehart algebra $\gPproj$ over $R$ with componentwise bracket and anchor
$(\xi,Y)\mapsto\xi$. A \emph{Poisson connection} is an $R$-linear
section $H$ of this anchor. We write $H_\xi$ for its vector-field
component. Its $A$-linear extension is a \emph{Poisson horizontal
splitting} of~\eqref{eq:tangent-sequence}, meaning that the lift of
$1\otimes\xi$ preserves $\pi$ for each $\xi\in T$.
\end{definition}

The \emph{curvature} is
$F_H(\xi,\eta)=[H_\xi,H_\eta]-H_{[\xi,\eta]}$.
The \emph{covariant exterior derivative} $d_H$ is given, for
$\alpha\in\Omega_T^p(A)$, by
\begin{align*}
(d_H\alpha)(\xi_0,\ldots,\xi_p)
&=\sum_i(-1)^i H_{\xi_i}\alpha(\xi_0,\ldots,\widehat\xi_i,\ldots,\xi_p)\\
&\quad+\sum_{i<j}(-1)^{i+j}
\alpha([\xi_i,\xi_j],\xi_0,\ldots,\widehat\xi_i,\ldots,
\widehat\xi_j,\ldots,\xi_p).
\end{align*}
For algebra-valued forms we extend the coefficient bracket by
$[\theta\otimes a,\eta\otimes b]=\theta\wedge\eta\otimes\{a,b\}$,
and use the same convention for the quantum bracket.
The square of $d_H$ is the action of the curvature $F_H$.

\begin{definition}[Hamiltonian unfolding]\label{def:geometric}
A \emph{Hamiltonian unfolding} is a pair $(H,C)$, where $H$ is a
Poisson connection and $C\in\Omega_T^2(A)$ is a two-form, called
the \emph{curving}, satisfying
\begin{equation}\label{eq:pair}
       F_H(\xi,\eta)=X_{C_{\xi\eta}}, \qquad d_H C=0.
\end{equation}
The unfolding is \emph{strict} if $C=0$, in which case $H$ is flat.
\end{definition}

The curvature and closure equations are the algebraic counterparts
of the geometric data for a coupling Dirac structure
\cite[Proposition~2.4]{BF}.

For fixed $(A,\pi)$ these unfoldings form a groupoid. Its arrows
are \emph{changes of splitting}. A one-form $b\in\Omega_T^1(A)$
defines an arrow from $(H,C)$ to the unfolding
\begin{equation}\label{eq:change}
 H'_\xi=H_\xi+X_{b_\xi}, \qquad
 C'_{\xi\eta}=C_{\xi\eta}+(d_H b)_{\xi\eta}+\{b_\xi,b_\eta\},
\end{equation}
where $(d_Hb)_{\xi\eta}=H_\xi b_\eta-H_\eta b_\xi-b_{[\xi,\eta]}$.
Composition adds the one-forms, so $0$ is the identity and $-b$ is
the inverse of $b$.

The extension description of Poisson actions up to homotopy
\cite[Section~1]{Actions} suggests how to combine transport and its
curvature Hamiltonians. Over the base algebra $R$, consider extensions
of Lie--Rinehart algebras
\begin{equation}\label{eq:ext}
 0\longrightarrow (A,\{-,-\})\xrightarrow{i} E\xrightarrow{p} T
 \longrightarrow0.
\end{equation}
The kernel carries its Poisson $R$-Lie algebra structure and zero anchor.
We equip $E$ with an action, namely a Lie--Rinehart morphism
$\rho:E\to\gPproj$ over $T$, satisfying
$\rho i(f)=(0,X_f)$ and $[e,i(f)]=i(\rho(e)f)$, where
$\rho(e)f$ means the action of the vector-field component.
Isomorphisms that fix $A$ and $T$ and intertwine $\rho$ make these
extensions with action into a groupoid.

\begin{proposition}[Reconstruction of Hamiltonian unfoldings]\label{thm:reconstruction}
The groupoid of Hamiltonian unfoldings of $(A,\pi)$ is equivalent to
the groupoid of extensions~\eqref{eq:ext} with action $\rho$.
Specifying a strict unfolding amounts to choosing a flat Poisson
horizontal splitting of~\eqref{eq:tangent-sequence}.
\end{proposition}
\begin{proof}
Consider an extension \eqref{eq:ext}. Since $T$ is projective over $R$,
the projection $p$ has an $R$-linear splitting $s$. Define $H=\rho s$ and
$C$ by $i(C_{\xi\eta})=[s\xi,s\eta]-s[\xi,\eta]$. Applying $\rho$ gives
$F_H=X_C$. The vertical component of Jacobi on $s\xi,s\eta,s\zeta$ is
$i((d_HC)_{\xi\eta\zeta})$, so injectivity of $i$ gives $d_HC=0$.
Replacing $s$ by $s+i b$ gives precisely \eqref{eq:change}.

Conversely give $A\oplus T$ the action $(f,\xi)\mapsto X_f+H_\xi$ and bracket
\[
 [(f,\xi),(g,\eta)]=
  \big(\{f,g\}+H_\xi g-H_\eta f+C_{\xi\eta},[\xi,\eta]\big).
\]
Jacobi on triples with zero, one, two and three horizontal entries follows
respectively from Poisson Jacobi, $\cL_H\pi=0$, $F_H=X_C$, and $d_H C=0$.
The identity $[H_\xi,X_f]=X_{H_\xi f}$ gives bracket preservation for the
action. The map $(f,\xi)\mapsto i(f)+s\xi$ recovers $E$. Every
isomorphism fixing both ends has the form
$(f,\xi)\mapsto(f-b_\xi,\xi)$, and preservation of
action and bracket gives exactly \eqref{eq:change}.
For a strict unfolding, \eqref{eq:pair} reduces to $F_H=0$, giving
the correspondence with flat Poisson horizontal splittings.
\end{proof}

Flatness also appears in the holomorphic setting, where transversal
unfoldings of a family of Lie algebroids correspond to flat connections
on an associated Lie algebroid over the base~\cite{CMQ}.

\begin{remark}[The choice of curving for a fixed connection]
\label{rem:classical-curvings}
Fix $H$ and choose curvature generators $C$ with $F_H=X_C$.
The curvature acts trivially on $Z_\pi$, so $H$ induces a flat
connection there. The Bianchi identity gives $X_{d_HC}=0$, and
$d_H^2C=[C,C]=0$ because $C$ has even form degree. Thus $d_HC$
defines a class in $H^3(T;Z_\pi)$.

Other generators differ from $C$ by a central two-form. A closed
curving therefore exists exactly when $[d_HC]=0$, the Poisson instance
of the extension obstruction in~\cite{BMRT}. Once such a curving is
chosen, the others differ from it by closed central two-forms.
An arrow fixing $H$ has central $b$ and adds $d_Hb$ to the curving.
The equivalence classes thus form a torsor for $H^2(T;Z_\pi)$, and
the automorphisms are the closed central one-forms. A \emph{torsor}
is a nonempty set with a free and transitive group action.
\end{remark}

\subsection{Lifting to a prescribed star product}\label{sec:lifts}
In \v Severa's quantisation of Poisson families, the product and
compatible transport are constructed together~\cite{Severa}. Here
we fix an unfolding $(H,C)$ and prescribe the product on $A[[\hbar]]$
in advance. Its coefficients depend algebraically on the base and
differentiate along the fibres in the following sense
\cite[Tag~09CH]{Stacks}.

\begin{definition}[Relative differential operators]
\label{def:relative-differential-operators}
For $r\ge0$, an $R$-linear map $U:A\to A$ is a
\emph{relative differential operator} of order at most $r$ if
\[
 [\cdots[[U,m_{f_0}],m_{f_1}],\ldots,m_{f_r}]=0
 \qquad(f_0,\ldots,f_r\in A),
\]
where $m_f(g)=fg$. We write $\operatorname{Diff}_R(A)$ for the
operators of finite order. An $R$-bilinear map $A\times A\to A$ is
\emph{bidifferential} if, with either argument fixed, it is a
differential operator in the other, with order bounds independent
of the fixed argument.
\end{definition}

\begin{definition}[Relative differential star product]
\label{def:star-product}
A \emph{relative differential star product} quantising $(A,\pi)$ is a
unital associative $R[[\hbar]]$-bilinear product
\[
 a\star b=ab+\sum_{n\ge1}\hbar^n B_n(a,b)
\]
on $A[[\hbar]]$, where each $B_n$ is a relative bidifferential
operator, $B_n(1,a)=B_n(a,1)=0$, and
$B_1(a,b)-B_1(b,a)=\{a,b\}$.
\end{definition}

Specialising at $s\in\Spec R$ gives a star product on
$(A\otimes_R\kappa(s))[[\hbar]]$ for the Poisson bracket of that fibre.
Write $A_\hbar=(A[[\hbar]],\star)$.
Divisibility of $[a,b]_\star$ by $\hbar$ and $\hbar$-torsion-freeness
give the \emph{divided commutator} $\mu(a,b)=[a,b]_\star/\hbar$,
which reduces to $\{a,b\}$. Write $I_a=\mu(a,-)$.
For a formal series, $[\hbar^n]$ denotes its coefficient of $\hbar^n$;
$O(\hbar^n)$ means divisibility by $\hbar^n$. The same notation
applies coefficientwise to forms and operators.

A \emph{derivation connection} on $A_\hbar$ is an $R$-linear assignment
$\xi\mapsto D_\xi$ of continuous $k[[\hbar]]$-linear derivations
such that $D_\xi(r)=\xi(r)$ for $r\in R$. Thus
$D_\xi(ra)=\xi(r)a+rD_\xi(a)$.
Its curvature $F_D$ and covariant exterior derivative $d_D$ are
defined as for $H$. For compatible connections on families of
symplectic star products see~\cite[Proposition~5.1]{AMS}.

\begin{definition}[Full Hamiltonian lift]\label{def:full-lift}
A \emph{full Hamiltonian lift} of $(H,C)$ to $A_\hbar$ is a pair
$(D,C_\hbar)$, where $D$ is a derivation connection and
$C_\hbar\in\Omega_T^2(A_\hbar)$, reducing modulo $\hbar$ to
$H$ and $C$, respectively, and satisfying
\begin{equation}\label{eq:full}
             F_D=I_{C_\hbar},\qquad d_D C_\hbar=0 .
\end{equation}
\end{definition}

\begin{remark}[The quantum extension bracket]\label{rem:quantum-extension}
A derivation connection $D$ and a two-form $C_\hbar$ give the formal
module $E_\hbar=(A\oplus T)[[\hbar]]$ an anchor $(a,\xi)\mapsto\xi$
and an alternating bracket
\[
 [(a,\xi),(b,\eta)]
 =\bigl(\mu(a,b)+D_\xi b-D_\eta a+C_{\hbar,\xi\eta},
                                      [\xi,\eta]\bigr).
\]
The Jacobi identities with two and three horizontal entries are,
respectively, $F_D=I_{C_\hbar}$ and $d_DC_\hbar=0$.
Those with at most one horizontal entry follow from associativity
of $\star$ and the derivation rule for $D$.
Thus a full Hamiltonian lift defines a Lie--Rinehart extension
of $T[[\hbar]]$ by $(A_\hbar,\mu)$, whose classical reduction
is the extension~\eqref{eq:ext}.
\end{remark}

In the quantum extension bracket, shifting horizontal lifts by
$b'\in\Omega_T^1(A_\hbar)$ gives
\begin{equation}\label{eq:qchange}
 D'_\xi=D_\xi+I_{b'_\xi},\qquad
 C'_{\xi\eta}=C_{\hbar,\xi\eta}+(d_D b')_{\xi\eta}+
                         \mu(b'_\xi,b'_\eta).
\end{equation}
In the groupoid of lifts of a fixed classical pair, the arrows are the
changes \eqref{eq:qchange} reducing to the \emph{identity arrow} of that
pair, so $b'\in\hbar\Omega_T^1(A_\hbar)$. As classically, composition
is addition and the inverse arrow is $-b'$. A one-form with arbitrary reduction transports
lifts between the corresponding classical splittings.

\begin{definition}[Relative differential equivalence]
\label{def:product-equivalence}
A \emph{relative differential equivalence} is a unital
$R[[\hbar]]$-linear map
$\Phi=\id+\sum_{j\ge1}\hbar^j\Phi_j$, with each $\Phi_j$
a relative differential operator. It identifies $\star$ with
\[
 f\star^\Phi g=\Phi\bigl(\Phi^{-1}(f)\star\Phi^{-1}(g)\bigr).
\]
\end{definition}

Its inverse is obtained recursively and again has differential
coefficients. It transports a Hamiltonian lift by
\[
 (D,C_\hbar)\longmapsto(\Phi D\Phi^{-1},\Phi C_\hbar),
\]
preserving the classical pair and the splitting-change formulas.

To solve the product Leibniz rule order by order we will need
Hochschild primitives that are themselves differential operators.
Relative Hochschild cochains are
$R$-multilinear maps from tensor powers of $A$ to $A$.
A cochain is \emph{normalised} if it vanishes whenever one of its
arguments is $1$. Our coboundary convention is
\[
                   (bU)(f,g)=fU(g)-U(fg)+U(f)g.
\]
A bilinear cochain $E$ is a \emph{Hochschild two-cocycle} when
\[
 fE(g,h)-E(fg,h)+E(f,gh)-E(f,g)h=0.
\]
This is precisely the associativity condition for
$f\mathbin{\circ}g=fg+\epsilon E(f,g)$ modulo $\epsilon^2$
\cite[Section~1.2]{KonFormality}.

\begin{lemma}[Differential Hochschild primitives]\label{lem:sym}
Let $A$ be smooth over $R$.
If an $R$-linear map $U:A\to A$ satisfies $U(1)=0$ and its
Hochschild coboundary $bU$ is bidifferential, then $U$ is a
relative differential operator. Every normalised symmetric relative
differential Hochschild two-cocycle $E$ is a coboundary $bU$ with
$U$ a normalised relative differential operator. 
\end{lemma}
\begin{proof}
For multiplication $m_f$ by $f$ one has
\[
 [U,m_f]=m_{U(f)}-(bU)(f,-).
\]
If $bU$ has order at most $r$ in its second argument, the
right-hand side has order at most $r$, uniformly in $f$.
Iterated commutators with multiplication therefore show that $U$
has order at most $r+1$.

For a symmetric cocycle $E$, the multiplication
$(f,u)(g,v)=(fg,fv+ug-E(f,g))$ defines a commutative square-zero
$R$-algebra extension of $A$ by $A$. Smoothness implies formal
smoothness, so the identity of its quotient $A$ lifts to an
$R$-algebra section $f\mapsto(f,U(f))$
\cite[Definition~10.138.1 and Proposition~10.138.13]{Stacks}.
Multiplicativity of this section gives $bU=E$, and its unit gives
$U(1)=0$.
The first assertion makes $U$ differential.
\end{proof}

Applied to the first coefficient of a star product, the lemma gives
a differential equivalence normalising $B_1$ to $\pi/2$. Indeed,
$E=B_1-\pi/2$ is a normalised symmetric cocycle, and
$\Phi=\id+\hbar U$ gives $B'_1=B_1-bU=\pi/2$.

The same lemma controls the coefficients of a derivation connection
$D=H+\sum_{n\ge1}\hbar^nD_n$. For any operator $Y$ and any bilinear
cochain $B$, write
\[
 (\cL_YB)(f,g)=Y(B(f,g))-B(Yf,g)-B(f,Yg).
\]
In order $n\ge1$ the Leibniz rule gives
\[
 bD_{n,\xi}=\cL_{H_\xi} B_n+
                 \sum_{i=1}^{n-1}\cL_{D_{i,\xi}}B_{n-i}.
\]
The first term is $R$-bilinear because $H_\xi$ lifts $\xi$. It
preserves the order bounds of $B_n$ in each argument. Indeed, taking
a commutator with multiplication by $f$ also differentiates $f$,
by $[H_\xi,m_f]=m_{H_\xi f}$, so the same iterated-commutator
identities apply. The other terms are bidifferential by induction.
Lemma~\ref{lem:sym} now gives differentiality of $D_{n,\xi}$.

For $r\ge0$, consider
\[
 D^{[r]}=H+\sum_{i=1}^{r}\hbar^iD_i,\qquad
 C^{[r]}=C+\sum_{i=1}^{r}\hbar^ic_i,
\]
where each $D_i$ is an $R$-linear one-form with values in relative
differential operators on $A$ vanishing on $1$, and
$c_i\in\Omega_T^2(A)$. In the curvature equation below, form $I_a$ in $A_\hbar$ before
reducing modulo a power of $\hbar$.

\begin{definition}[Partial Hamiltonian lift]\label{def:partial-lift}
The coefficients of $(D^{[r]},C^{[r]})$ form a \emph{partial
Hamiltonian lift through order $r$} if
\[
 F_{D^{[r]}}-I_{C^{[r]}}\equiv0,\qquad
 d_{D^{[r]}}C^{[r]}\equiv0
                    \pmod{\hbar^{r+1}},
\]
and an auxiliary coefficient $D_{r+1}$ can be chosen so that
$\widehat D=D^{[r]}+\hbar^{r+1}D_{r+1}$ satisfies
\[
 \widehat D_\xi(a\star b)
 -\widehat D_\xi(a)\star b-a\star\widehat D_\xi(b)
                    \equiv0\pmod{\hbar^{r+2}}
\]
for every $\xi\in T$ and $a,b\in A$. The auxiliary coefficient
has the same differential-operator and $R$-linearity properties as
the $D_i$.
\end{definition}

The partial lift records the coefficients through order $r$. The
extra order in the product equation is needed because the Leibniz
rule for $\mu$ is obtained by antisymmetrising and dividing by $\hbar$.

\subsection{The central Bianchi obstruction}\label{sec:central-bianchi}
Fix a derivation connection $D$ on the prescribed product lifting $H$.
Suppose we have a \emph{preliminary curving}
$\widetilde C_\hbar\in\Omega_T^2(A_\hbar)$ reducing to $C$ and satisfying
$F_D=I_{\widetilde C_\hbar}$. We seek a correction that closes this
two-form while preserving its curvature and classical reduction.

Write $Z_\hbar=Z(A_\hbar)$. If $\hbar a$ is central,
$\hbar[a,-]_\star=0$ implies $a$ is central by torsion-freeness.
Thus $Z_\hbar\cap\hbar A_\hbar=\hbar Z_\hbar$. Derivations preserve
the centre, and the curvature $I_{\widetilde C_\hbar}$ acts trivially
on it. Hence $D$ restricts to a flat connection on $Z_\hbar$ and on
$\hbar Z_\hbar$, with cochain differential $d_D$.

\begin{definition}[Bianchi defect]\label{def:bianchi-defect}
The \emph{Bianchi defect} of the preliminary curving
$\widetilde C_\hbar$ is the three-form
$d_D\widetilde C_\hbar\in\Omega_T^3(A_\hbar)$.
Its vanishing says that $(D,\widetilde C_\hbar)$ is a full
Hamiltonian lift.
\end{definition}

\begin{theorem}[Bianchi obstruction for a fixed connection]\label{thm:central-bianchi}
The Bianchi defect determines a well-defined class
\[
 \mathfrak b(D,C)=[d_D\widetilde C_\hbar]
       \in H^3(T;\hbar Z_\hbar),
\]
independent of the preliminary curving. It vanishes exactly when $D$
admits a closed quantum curving reducing to $C$.
\end{theorem}
\begin{proof}
The Bianchi identity for the derivation connection $D$ gives
\[
 0=d_DF_D=d_D I_{\widetilde C_\hbar}
   =I_{d_D\widetilde C_\hbar}.
\]
Consequently $d_D\widetilde C_\hbar$ is centre-valued. Its reduction is
$d_HC=0$, so it belongs to $\Omega_T^3(\hbar Z_\hbar)$. Moreover,
\[
 d_D^2\widetilde C_\hbar
     =F_D\cdot\widetilde C_\hbar
     =[\widetilde C_\hbar,\widetilde C_\hbar]_{\star}/\hbar=0.
\]
The last bracket vanishes because $\widetilde C_\hbar$ has even
form degree. Thus the Bianchi defect is closed.

Two preliminary curvings differ by a two-form $z$ with $I_z=0$
and zero classical reduction. By saturation these are exactly the
forms in $\Omega_T^2(\hbar Z_\hbar)$. Adding $z$ changes the defect
by $d_Dz$, proving independence of the class. Its vanishing is
equivalent to finding such a $z$ with
$d_Dz=-d_D\widetilde C_\hbar$. Then $\widetilde C_\hbar+z$ is the
required closed curving.
\end{proof}

For fixed $D$, we classify the closed curvings up to changes of
splitting~\eqref{eq:qchange} reducing to the classical identity.

\begin{corollary}[Closed curvings for a fixed connection]
\label{cor:central-curvings}
If $\mathfrak b(D,C)=0$, the equivalence classes of closed quantum
curvings reducing to $C$ form a torsor for $H^2(T;\hbar Z_\hbar)$. Each curving has
automorphism group $Z^1(T;\hbar Z_\hbar)$.
\end{corollary}
\begin{proof}
Fix one such $C_\hbar$. Every other solution differs from it by a closed
two-form in $\Omega_T^2(\hbar Z_\hbar)$. A change of splitting by
$b\in\hbar\Omega_T^1(A_\hbar)$ fixes $D$ precisely when $I_b=0$, that is,
when $b\in\Omega_T^1(\hbar Z_\hbar)$. Formula~\eqref{eq:qchange} then reduces
to $C_\hbar\mapsto C_\hbar+d_Db$. This proves the torsor statement and
identifies strict automorphisms with closed central one-forms.
\end{proof}

A change of splitting by $b$ replaces $D$ by $D+I_b$ and
$\widetilde C_\hbar$ by $\widetilde C_\hbar+d_Db+\tfrac12[b,b]_\mu$,
where $[-,-]_\mu$ is the shuffle bracket of forms. Expanding the new
covariant derivative gives the same Bianchi defect. Indeed,
$d_D^2b=[\widetilde C_\hbar,b]_\mu$ cancels $[b,\widetilde C_\hbar]_\mu$,
$d_D(\tfrac12[b,b]_\mu)=[d_Db,b]_\mu$ cancels $[b,d_Db]_\mu$,
and $[b,[b,b]_\mu]_\mu=0$ by graded Jacobi. The differential on the
centre also agrees since $I_b$ acts trivially there. This identifies
the obstruction classes for the corresponding classical splittings.

The classical coefficients available for testing vanishing are the
Casimirs that lift to the quantum centre. Indeed, saturation embeds
$\overline Z_\hbar=Z_\hbar/\hbar Z_\hbar$ in $A$, and reduction of
the divided commutator places it in $Z_\pi$. It carries the
connection induced by $D$.

\begin{theorem}[A quantum-centre vanishing criterion]\label{thm:central-vanishing}
Let $D$ be a derivation connection lifting $H$, and suppose that
$\widetilde C_\hbar\in\Omega_T^2(A_\hbar)$ reduces to $C$ and satisfies
$F_D=I_{\widetilde C_\hbar}$.
If $H^3(T;\overline Z_\hbar)=0$, then $H^3(T;\hbar Z_\hbar)=0$.
In particular, the fixed connection $D$ admits a closed quantum
curving reducing to $C$.
\end{theorem}
\begin{proof}
The centre is $\hbar$-adically closed, being the intersection of
the kernels of the continuous commutator maps. Together with
saturation, this makes $Z_\hbar$ complete and separated for its
own $\hbar$-adic topology.
Set $M=\hbar Z_\hbar$. Since $T$ is finite projective, the complex
$\Omega_T^\bullet(M)$ is complete and separated in the $\hbar$-adic
topology, and its reduction modulo $\hbar$ is
$\Omega_T^\bullet(\overline Z_\hbar)$. Let $\gamma$ be a closed
three-cochain in $\Omega_T^\bullet(M)$. The assumed third-cohomology
vanishing supplies a two-cochain $b_0$ whose differential equals
$\gamma$ modulo $\hbar M$. The residual cocycle lies in
$\hbar\Omega_T^3(M)$. Dividing by $\hbar$ and repeating, we obtain
$b_n\in\Omega_T^2(M)$ such that
\[
  \gamma-d_D(b_0+\hbar b_1+\cdots+\hbar^n b_n)
              \in\hbar^{n+1}\Omega_T^3(M).
\]
The series of corrections converges and gives $\gamma=d_Db$.
Therefore $H^3(\Omega_T^\bullet(M))=0$, and
Theorem~\ref{thm:central-bianchi} applies.
\end{proof}

Three cases illustrate the criterion for the fixed connection $D$.
\begin{enumerate}
\item \emph{Every Casimir has a central lift.}
Here $\overline Z_\hbar=Z_\pi$ with the connection induced by
$H$, so $H^3(T;Z_\pi)=0$ suffices to correct the preliminary curving.

\item \emph{The base has dimension at most two.}
Since $\wedge_R^3T=0$, the Bianchi defect itself vanishes and the
preliminary curving is already closed.

\item \emph{The fibres are geometrically connected and symplectic.}
Symplectic nondegeneracy gives
$Z_\pi=K_{A/R}=\ker(d_{\mathrm v}:A\to\Omega^1_{A/R})$.
In fact $K_{A/R}=R$: a relative constant is constant on each
geometric fibre, so its two pullbacks to $A\otimes_R A$ agree
on all geometric fibres. This tensor product is reduced, being
smooth over the smooth algebra $R$; hence the pullbacks agree,
and the faithfully flat equalizer property gives the equality
\cite[Tag~023M]{Stacks}.

Every central quantum series consequently has reduction in $R$.
Subtracting this scalar and dividing by $\hbar$ gives another
central series. Induction yields $Z_\hbar=R[[\hbar]]$, on which
$D$ acts coefficientwise by base derivations. Therefore
\[
 H^j\bigl(\Omega_T^\bullet(\hbar Z_\hbar),d_D\bigr)
       \simeq \hbar H^j_{\mathrm{dR}}(\Spec R)[[\hbar]],
\]
and $H^3_{\mathrm{dR}}(\Spec R)=0$ suffices. Here algebraic
de~Rham cohomology is computed by the global complex
$(\Omega^\bullet_{R/k},d)$ of the smooth affine base.
\end{enumerate}

\Needspace{10\baselineskip}
\section{Reductive Lie algebra bundles}\label{sec:reductive}

For which families does a given unfolding lift to \emph{every}
prescribed product? On the dual of a reductive Lie algebra bundle,
linear observables come from its semisimple part and its centre.
We lift the semisimple generators first and put transport on them
in a prescribed form. The remaining equations then lie in their
centraliser. Its structure will give the sharp criterion in terms
of the rank of the centre, including for bundles with monodromy.
\medskip

\noindent
Subsection~3.1 normalises the classical unfolding along the
semisimple summand, lifts its generators, and establishes a
factorisation of the quantum algebra through their centraliser.
Subsection~3.2 uses this factorisation to reduce the lifting
problem to the centraliser and proves compatibility with
\'etale descent. In Subsection~3.3 we obtain the criterion
in terms of the rank of the centre. We also construct lifts
for the Poincar\'e--Birkhoff--Witt product in every central rank
and exhibit obstructions at any prescribed order when the
centre has rank at least two. Subsection~3.4 illustrates the
global construction on a bundle with a nontrivial centre line
and monodromy in its semisimple summand.

\subsection{Normalisation and factorisation}
To compare classical and quantum transport, we first construct a
reference connection on the semisimple part of the bundle.

Let $k=\mathbb C$ and let $R$ be a smooth affine $\mathbb C$-algebra.
A \emph{reductive Lie algebra bundle of locally constant type} is a
finite projective $R$-module $\mathfrak g$ with an $R$-bilinear Lie
bracket such that, for a faithfully flat \'etale $R$-algebra $R'$,
\[
 \mathfrak g\otimes_R R'\simeq\mathfrak g_*\otimes_{\mathbb C}R'
\]
as Lie algebras, where $\mathfrak g_*$ is a fixed finite-dimensional
complex reductive Lie algebra. The derived algebra and centre of
$\mathfrak g_*$ are preserved by every automorphism. Their decomposition
therefore descends to a decomposition into finite projective $R$-modules
\[
 \mathfrak g=\mathfrak s\oplus\mathfrak z,\qquad
 \mathfrak s=[\mathfrak g,\mathfrak g],\qquad
 \mathfrak z=Z(\mathfrak g).
\]
The coordinate algebra of the dual bundle $\mathfrak g^*$ is
$A=\Sym_R\mathfrak g$. Its \emph{linear Poisson bracket} is determined
by
\[
 \{x,y\}=[x,y]\quad(x,y\in\mathfrak g),\qquad
 \{r,a\}=0\quad(r\in R,\ a\in A),
\]
and the Leibniz rule in each argument. The elements of $\mathfrak z$
are central linear functions. The full algebra of polynomial
Casimirs is
\[
 Q=A^{\mathfrak s}
   =Q_s\otimes_R\Sym_R\mathfrak z,\qquad
 Q_s=(\Sym_R\mathfrak s)^{\mathfrak s}.
\]
Here superscript $\mathfrak s$ means invariance under the Hamiltonian
action $x\cdot a=\{x,a\}$. Invariance under $\mathfrak s$ gives the
Poisson centre because $\mathfrak s\oplus\mathfrak z$ generates $A$.

Fix a Hamiltonian unfolding $(H,C)$ and a prescribed relative
differential star product $A_\hbar$ quantising this bracket.
We also use the completed enveloping algebra
\[
 U=\left(
 T_R(\mathfrak s)[\hbar]\big/
 (xy-yx-\hbar[x,y]:x,y\in\mathfrak s)
 \right)^{\wedge}_{(\hbar)}.
\]
Here $T_R(\mathfrak s)=\bigoplus_{d\ge0}\mathfrak s^{\otimes_R d}$
is the tensor algebra, and the superscript $\wedge_{(\hbar)}$
denotes completion in powers of $\hbar$.
The Poincar\'e--Birkhoff--Witt (PBW) theorem for projective Lie
algebras over a ring~\cite[Theorem~3.1]{Rinehart} makes symmetrisation
the continuous
$R[[\hbar]]$-module isomorphism
\[
 \operatorname{sym}:(\Sym_R\mathfrak s)[[\hbar]]
       \xrightarrow{\ \sim\ }U,\qquad
 x_1\cdots x_d\longmapsto
       \frac1{d!}\sum_{\sigma\in S_d}x_{\sigma(1)}\cdots x_{\sigma(d)}.
\]
Write $Z_U=Z(U)$ for its associative centre. The relation between
linear Poisson deformations and rigidity of enveloping algebras is
studied in~\cite{BMP}.

We next extract a connection on $\mathfrak s$ from $H$.
Let $\operatorname{pr}_1:A\to\mathfrak g$ be projection to symmetric
degree one and $\operatorname{pr}_{\mathfrak s}:\mathfrak g\to\mathfrak s$
the reductive projection. Set
\[
 \nabla_\xi x
   =\operatorname{pr}_{\mathfrak s}\operatorname{pr}_1(H_\xi x),
 \qquad \xi\in T,\quad x\in\mathfrak s.
\]
The Leibniz rule for $H$ gives
$\nabla_\xi(rx)=\xi(r)x+r\nabla_\xi x$.
Bracketing with a linear function preserves symmetric degree.
Projecting the Poisson derivation identity for $H_\xi$ therefore gives
\[
 \nabla_\xi[x,y]=[\nabla_\xi x,y]+[x,\nabla_\xi y].
\]
A connection satisfying this identity is called a \emph{bracket connection}.

Every derivation of a complex semisimple Lie algebra is inner,
and its centre is zero~\cite[Proposition~17.9]{EtingofLie}.
Applying this in an \'etale trivialisation and descending gives
$\ad:\mathfrak s\xrightarrow{\sim}\Der_R(\mathfrak s)$, where
$\Der_R(\mathfrak s)$ denotes Lie algebra derivations. Consequently
there is a unique $\mathfrak s$-valued two-form $\Omega$ such that
\[
 F_\nabla(\xi,\eta)
   =[\nabla_\xi,\nabla_\eta]-\nabla_{[\xi,\eta]}
   =\ad_{\Omega_{\xi\eta}}.
\]
The connection Bianchi identity gives
$\ad(d_\nabla\Omega)=0$, hence $d_\nabla\Omega=0$.
Preservation of the bracket allows $\nabla$ to extend, by the
Leibniz rule, to a derivation connection $\nabla^U$ on $U$.
Its curvature is
$F_{\nabla^U}=I_\Omega$, with $I_\Omega=[\Omega,-]_U/\hbar$.
Thus $\nabla^U$ is flat on $Z_U$, and $\Omega$ generates the
semisimple curvature.

\begin{definition}[Normalised representative]\label{def:normalised-pair}
A representative $(H_0,C_0)$ obtained from $(H,C)$ by a change of
splitting is \emph{normalised relative to $\nabla$} if
$H_{0,\xi}(x)=\nabla_\xi x$ for every $x\in\mathfrak s$.
Its \emph{residual curving} is $\overline C=C_0-\Omega$.
\end{definition}

\begin{definition}[Quantum lift of the semisimple generators]
\label{def:semisimple-generators}
A \emph{quantum lift of the semisimple generators} is an $R$-linear map
$j:\mathfrak s\to A_\hbar$ such that
\[
 j(x)\equiv x\pmod\hbar,\qquad
 [j(x),j(y)]_\star=\hbar j([x,y]).
\]
It extends to a continuous algebra map $\widehat j:U\to A_\hbar$.
Its \emph{centraliser} is the closed subalgebra
\[
 B=\operatorname{Cent}_{A_\hbar}(j(\mathfrak s))
   =\{b\in A_\hbar:[j(x),b]_\star=0\text{ for all }x\in\mathfrak s\}.
\]
\end{definition}

The map $\widehat j$ realises the action $x\mapsto I_{j(x)}$
through the enveloping algebra, as in the definition of a quantum
momentum map~\cite[Definition~6.2]{Xu}. Since $B$ commutes with
$\widehat j(U)$, it is an algebra over $Z_U$. The table collects the
notation for the classical centres and the two quantum algebras.
\begin{center}
\renewcommand{\arraystretch}{1.12}
\begin{tabular}{@{}p{.14\textwidth}p{.79\textwidth}@{}}\toprule
$\mathfrak z$ & Lie-algebra centre: linear Casimirs from the abelian summand.\\
$Q_s$, $Q$ & Polynomial Poisson centres of the semisimple and reductive algebras.\\
$Z_U$ & Associative centre of the completed semisimple enveloping algebra.\\
$B$ & Associative centraliser of the chosen quantum semisimple generators.\\
\bottomrule\end{tabular}
\end{center}
All completed tensor products use the $\hbar$-adic filtrations
of these algebras. In particular, with
$U^{(n)}=U/\hbar^nU$, $Z_U^{(n)}=Z_U/\hbar^nZ_U$ and $B^{(n)}=B/\hbar^nB$,
\[
 U\widehat\otimes_{Z_U}B
   =\varprojlim_n\bigl(U^{(n)}\otimes_{Z_U^{(n)}}B^{(n)}\bigr).
\]
\begin{lemma}[Normalisation and quantum factorisation]\label{lem:quantum-factor}
A normalised classical representative and a quantum lift $j$ exist.
For every such $j$, reduction and multiplication induce isomorphisms
of algebras
\begin{equation}\label{eq:factor}
 B/\hbar B\simeq Q,\qquad
 U\widehat\otimes_{Z_U}B\xrightarrow{\ \sim\ }A_\hbar.
\end{equation}
For a normalised representative, $H_0|_Q$ is flat and
$\overline C=C_0-\Omega$ is an $H_0$-closed $Q$-valued two-form.
\end{lemma}
\begin{proof}
Both normalisation and lifting use the cohomology of $\mathfrak s$
acting on $A$ by $x\cdot a=\{x,a\}$. Write
\[
 C_R^q(\mathfrak s,A)=
 \Hom_R\bigl(\wedge_R^q\mathfrak s,A\bigr)
\]
for its relative Chevalley--Eilenberg complex, with differential $\delta$.
We will use
\begin{equation}\label{eq:whitehead-relative}
 H_R^1(\mathfrak s,A)=H_R^2(\mathfrak s,A)=0.
\end{equation}
To prove this, pass to a faithfully flat \'etale trivialisation of
$\mathfrak g$. Each fixed symmetric-degree piece of $A$ then becomes
the scalar extension of a finite-dimensional module for the complex
semisimple fibre of $\mathfrak s$. Whitehead's lemmas give
vanishing in degrees one and two for each piece
\cite{CE}; see also \cite[Proposition~48.1(ii)]{EtingofLie}. Since
$\mathfrak s$ is finite projective, the cochain complex commutes
both with this flat base change and with the direct sum over
symmetric degrees. Flat base change is exact, so the cohomology
over $R$ vanishes by faithful flatness. This proves~\eqref{eq:whitehead-relative}.

We first apply this vanishing to the classical connection.
For $\xi\in T$, the map
\[
 c_\xi(x)=H_\xi(x)-\nabla_\xi x,
 \qquad x\in\mathfrak s,
\]
is $R$-linear. The two bracket-derivation identities imply
$c_\xi([x,y])=\{x,c_\xi(y)\}-\{y,c_\xi(x)\}$,
so $c_\xi$ is a one-cocycle. By~\eqref{eq:whitehead-relative},
there is $b_\xi\in A$ with $c_\xi(x)=\{b_\xi,x\}$.
The map $\xi\mapsto c_\xi$ is $R$-linear; projectivity of $T$
allows the primitives $b_\xi$ to be chosen $R$-linearly as well.
The change of splitting by $-b$ therefore gives a representative
$(H_0,C_0)$ with $H_0|_{\mathfrak s}=\nabla$.

Its curvature on a generator $x\in\mathfrak s$ is both
$\{C_{0,\xi\eta},x\}$ and $[\Omega_{\xi\eta},x]$.
Consequently $\overline C=C_0-\Omega$ has coefficients in $Q$.
The derivations $H_{0,\xi}$ preserve the Poisson centre $Q$,
and their curvature acts there by $\{C_0,-\}=0$; thus the restricted
connection is flat. Finally,
\[
 d_{H_0}\overline C
 =d_{H_0}C_0-d_{H_0}\Omega
 =0-d_\nabla\Omega=0.
\]

To lift the generators to the prescribed product, start with the inclusion $j_1:\mathfrak s\hookrightarrow A\subset A_\hbar$,
which preserves the divided bracket modulo $\hbar$.
Suppose an $R$-linear map $j_n$ preserves it modulo $\hbar^n$.
Its leading defect is
\[
 e_n(x,y)=[\hbar^n]\bigl(\mu(j_nx,j_ny)-j_n[x,y]\bigr).
\]
The coefficient of $\hbar^n$ in Jacobi says $\delta e_n=0$.
Replacing $j_n$ by $j_n+\hbar^nt_n$ changes this defect to
$e_n+\delta t_n$. Equation~\eqref{eq:whitehead-relative}
provides $t_n\in\Hom_R(\mathfrak s,A)$ with $\delta t_n=-e_n$.
The corrected map preserves brackets modulo $\hbar^{n+1}$.
Induction gives a convergent sequence of compatible maps and hence
a Lie map $j:\mathfrak s\to(A_\hbar,\mu)$ with the prescribed reduction.
The enveloping relations extend it to a continuous algebra map
$\widehat j:U\to A_\hbar$.

We next determine which classical functions lift to the centraliser.
Fix $v\in Q$ and suppose $v_n\equiv v\pmod\hbar$ commutes
with $j(\mathfrak s)$ for the divided bracket modulo $\hbar^n$.
The map $\epsilon_n(x)=[\hbar^n]\mu(jx,v_n)$ is a one-cocycle, again
by Jacobi. Equation~\eqref{eq:whitehead-relative} supplies
$a_n\in A$ with $\{x,a_n\}=-\epsilon_n(x)$.
The correction $v_{n+1}=v_n+\hbar^na_n$ then commutes
modulo $\hbar^{n+1}$. Starting with $v_1=v$ and passing to the limit
produces a lift in $B$.

Thus reduction $B\to Q$ is surjective. Moreover $B$ is the
intersection of the kernels of the continuous operators $I_{j(x)}$,
so it is closed in $A_\hbar$. If $\hbar^na\in B$, then
$\hbar^nI_{j(x)}a=0$ for every $x$, and torsion-freeness gives
$a\in B$. Hence
\[
 B\cap\hbar^nA_\hbar=\hbar^nB.
\]
This identifies the quotient $B/\hbar B$ with $Q$ and shows that
$B$ is complete and separated for its own $\hbar$-adic topology.

It remains to identify the multiplication map. We first give $A_\hbar$
coordinates as a complete module. Projectivity of $\mathfrak z$ supplies an $R$-linear lift
$\ell:\mathfrak z\to B$ of its inclusion in $Q$.
Symmetrised products give an $R$-linear map
$w:\Sym_R\mathfrak z\to B$:
\[
 w(z_1\cdots z_d)=\frac1{d!}\sum_{\sigma\in S_d}
 \ell(z_{\sigma(1)})\star\cdots\star\ell(z_{\sigma(d)}),
 \qquad w(1)=1.
\]
Put $P=(\Sym_R\mathfrak z)[[\hbar]]$ and extend $w$ continuously
to $P$. Multiplication defines a map of complete modules
\[
 \Theta:U\widehat\otimes_{R[[\hbar]]}P\longrightarrow A_\hbar,
 \qquad \Theta(u\otimes v)=\widehat j(u)\star w(v).
\]
Its reduction fits into the diagram
\[
\begin{CD}
 U\widehat\otimes_{R[[\hbar]]}P @>{\Theta}>> A_\hbar\\
 @VV{\bmod\hbar}V @VV{\bmod\hbar}V\\
 \Sym_R\mathfrak s\otimes_R\Sym_R\mathfrak z
 @>{\simeq}>> A.
\end{CD}
\]
By PBW the source is
$(\Sym_R\mathfrak s\otimes_R\Sym_R\mathfrak z)[[\hbar]]$ as a
complete module. Both sides are $\hbar$-torsion-free and separated.
The bottom arrow is the symmetric-algebra decomposition, so its
inverse lifts recursively in $\hbar$ to an inverse for $\Theta$.
In particular, $\widehat j$ embeds $U$ in $A_\hbar$.

Under these coordinates, $I_{j(x)}$ acts on the $U$ factor by
$[x,-]/\hbar$ and acts trivially on $P$, because $w(P)\subset B$.
PBW symmetrisation intertwines the divided adjoint action on $U$
with the classical adjoint action on $(\Sym_R\mathfrak s)[[\hbar]]$;
this is the equivariance underlying
\cite[Corollaries~13.7--13.8]{EtingofLie}, applied in local frames.
Its invariant submodule is exactly $Z_U$.
To take invariants in the tensor product, work coefficientwise in
$\hbar$. The module $\Sym_R\mathfrak z$ is flat, and invariants
are the kernel of the action map to $\Hom_R(\mathfrak s,-)$.
Finite projectivity of $\mathfrak s$ therefore gives the
$Z_U$-linear isomorphism
\[
 \Theta_B:Z_U\widehat\otimes_{R[[\hbar]]}P\xrightarrow{\ \sim\ }B,
 \qquad \Theta_B(z\otimes v)=\widehat j(z)\star w(v).
\]

The subalgebras $\widehat j(U)$ and $B$ commute, so balanced
multiplication gives an algebra map
$m:U\widehat\otimes_{Z_U}B\to A_\hbar$.
To identify its source, put $R^{(n)}=R[\hbar]/(\hbar^n)$ and
$P^{(n)}=(\Sym_R\mathfrak z)\otimes_RR^{(n)}$.
Reducing the complete-module isomorphism $\Theta_B$ gives
\[
 B^{(n)}\simeq Z_U^{(n)}\otimes_{R^{(n)}}P^{(n)},\qquad
 U^{(n)}\otimes_{Z_U^{(n)}}B^{(n)}\simeq U^{(n)}\otimes_{R^{(n)}}P^{(n)}.
\]
Taking inverse limits identifies $m$ with $\Theta$, proving
\eqref{eq:factor}.
\end{proof}

\subsection{Transport through the centraliser}
Fix a normalised unfolding $(H_0,C_0)$ and a generator lift $j$.
In~\eqref{eq:factor}, transport on $U$ is prescribed by $\nabla^U$.
On $B$ we must lift the flat connection $H_0|_Q$ and the residual
curving $\overline C$.

\begin{definition}[Centraliser lift]\label{def:centraliser-lift}
A \emph{centraliser lift} is a pair $(D^B,C^B)$, where $D^B$
is a derivation connection on $B$ and $C^B\in\Omega_T^2(B)$,
reducing to $(H_0|_Q,\overline C)$ and satisfying
\[
 D^B|_{Z_U}=\nabla^U|_{Z_U},\qquad
 F_{D^B}=I_{C^B},\qquad d_{D^B}C^B=0.
\]
We use the continuity and base Leibniz rule of
Definition~\ref{def:full-lift}, with $I_b(v)=[b,v]_\star/\hbar$
for $b,v\in B$.
The arrows are the splitting changes~\eqref{eq:qchange} with
one-forms in $\hbar\Omega_T^1(B)$.
\end{definition}

Saturation gives $[b,v]_\star/\hbar\in B$. Agreement on $Z_U$
will allow the two connections to act on the balanced tensor product.

\begin{theorem}[Global centraliser reduction]\label{thm:global}
The groupoid of Hamiltonian lifts of $(H,C)$ is equivalent to the
groupoid of centraliser lifts. For $(H_0,C_0)$, restriction and extension
through~\eqref{eq:factor} identify the latter with the full subgroupoid
of lifts satisfying $D_\xi j(x)=j(\nabla_\xi x)$ for $x\in\mathfrak s$
and $\xi\in T$.
\end{theorem}
\begin{proof}
Choose a formal lift of the one-form taking $(H,C)$ to $(H_0,C_0)$.
The splitting change~\eqref{eq:qchange} identifies their lift groupoids.
We first show that every lift $(D,C_\hbar)$ of $(H_0,C_0)$ is isomorphic
to one whose connection has the prescribed action on $j(\mathfrak s)$.
For each $\xi\in T$, consider the defect
\[
 e_\xi(x)=D_\xi j(x)-j(\nabla_\xi x),\qquad x\in\mathfrak s.
\]
The base Leibniz rules make this map $R$-linear in $x$ and $\xi$.
The derivation identities for $D$ and $\nabla$ give
\[
 e_\xi([x,y])
   =\mu(j(x),e_\xi(y))-\mu(j(y),e_\xi(x)).
\]
Also $e_\xi\equiv0\pmod\hbar$, since the classical pair is normalised.
If the defect is zero modulo $\hbar^n$, its leading coefficient
is thus a classical one-cocycle. By~\eqref{eq:whitehead-relative}
and projectivity of $T$, choose $a_n\in\Omega_T^1(A)$ with
\[
 [\hbar^n]e_\xi(x)=\{x,a_n(\xi)\}.
\]
Changing splitting by $\hbar^na_n$ replaces $D_\xi$ by
$D_\xi+\hbar^n I_{a_n(\xi)}$ and cancels this coefficient.
The accompanying change of curving preserves the lifting equations.
Iterating gives a convergent splitting change after which
\[
                 D_\xi j(x)=j(\nabla_\xi x).
\]
The Leibniz rule now gives $D|_U=\nabla^U$. To obtain the residual
connection, take $v\in B$. Differentiating
$\mu(j(x),v)=0$ gives
\[
 \mu(j(x),D_\xi v)=-\mu(j(\nabla_\xi x),v)=0,
\]
so $D_\xi v\in B$. Thus $D^B=D|_B$ is a derivation connection.
On $j(\mathfrak s)$ its curvature satisfies
\[
 I_{C_{\hbar,\xi\eta}}j(x)
    =F_D(\xi,\eta)j(x)
    =j([\Omega_{\xi\eta},x])
    =I_{j(\Omega_{\xi\eta})}j(x).
\]
Consequently $C^B=C_\hbar-j(\Omega)$ is $B$-valued.
Moreover $d_Dj(\Omega)=j(d_\nabla\Omega)=0$, so $d_{D^B}C^B=0$.
The term $j(\Omega)$ commutes with $B$; hence
$F_{D^B}=I_{C^B}$. Reduction gives $(H_0|_Q,\overline C)$,
and $D^B|_{Z_U}=\nabla^U|_{Z_U}$ since $Z_U\subset U$.

Conversely, given a centraliser lift $(D^B,C^B)$, put
\[
 \widetilde D_\xi(u\otimes v)
     =\nabla^U_\xi(u)\otimes v+u\otimes D^B_\xi(v).
\]
The equality of the connections on $Z_U$ makes this expression
respect the balancing relation $uz\otimes v=u\otimes zv$.
Transport it through multiplication $m$ in~\eqref{eq:factor}:
\[
\begin{CD}
 U\widehat\otimes_{Z_U}B @>{m}>> A_\hbar\\
 @V{\widetilde D_\xi}VV @VV{D_\xi}V\\
 U\widehat\otimes_{Z_U}B @>{m}>> A_\hbar .
\end{CD}
\]
Set $C_\hbar=j(\Omega)+C^B$. On $U$, the curvature of $D$ is
$I_{j(\Omega)}$, and on $B$ it is $I_{C^B}$. Since the two
subalgebras commute and topologically generate $A_\hbar$, continuity
gives $F_D=I_{C_\hbar}$ on the whole algebra. Both summands of the
curving have zero covariant differential.
Modulo $\hbar$, the connection agrees with $H_0$ on
$\mathfrak s$ and on $Q$. These generate $A$ over $R$ because
$\mathfrak z\subset Q$. The curving reduces to
$\Omega+\overline C=C_0$. The reconstructed connection is a
derivation of the prescribed differential product, so the induction
following Lemma~\ref{lem:sym} gives relative differential operators
for all its positive-order coefficients.

Finally, a change of splitting $b\in\hbar\Omega_T^1(A_\hbar)$
preserves the prescribed action on $j(\mathfrak s)$ exactly when
$I_{b_\xi}j(x)=0$ for every $x,\xi$, that is, when each $b_\xi$ lies
in $B$. Since $B\cap\hbar A_\hbar=\hbar B$, these are exactly the
one-forms in $\hbar\Omega_T^1(B)$. Formula~\eqref{eq:qchange}
restricts to the splitting change of the residual curving, and
composition is addition on both sides. Restriction and extension
are therefore inverse on this full subgroupoid. Its inclusion is an
equivalence by the normalisation established at the start of the proof.
\end{proof}

\begin{remark}[Changing the auxiliary choices]\label{rem:centraliser-choices}
If two generator lifts $j,j'$ agree modulo $\hbar^n$, their
leading difference is a one-cocycle with values in $A$.
Equation~\eqref{eq:whitehead-relative} expresses it as an inner
derivation. An automorphism $\exp(\hbar^n I_a)$, with the sign
of $a$ chosen accordingly, identifies the generators to the next
order. The successive automorphisms converge to a continuous
automorphism of $A_\hbar$ equal to the identity modulo $\hbar$.
Each coefficient is differential, being a finite combination of
coefficient operators of the quantum inner derivations.
Conjugation by this automorphism identifies the two centralisers
and transports their lift data. Changing the bracket connection
or classical normalisation likewise gives equivalent descriptions
through the groupoid of lifts of $(H,C)$.
\end{remark}

To compare lifts in different \'etale frames, keep the global
choices above and let $R\to R'$ be faithfully flat \'etale. Then
$\Der_{\mathbb C}(R')\simeq T\otimes_RR'$. Set
\[
 A'_\hbar=A_\hbar\widehat\otimes_RR',\qquad
 B'=\operatorname{Cent}_{A'_\hbar}
                      \bigl(j(\mathfrak s\otimes_RR')\bigr).
\]

A \emph{descent datum} consists of a lift over $R'$ and an arrow
between its two pullbacks to $R'\otimes_RR'$, satisfying the
cocycle identity over $R'\otimes_RR'\otimes_RR'$.
Morphisms of descent data commute with these arrows.

\begin{proposition}[\'Etale descent of lift data]\label{prop:reductive-etale-descent}
Base change induces an isomorphism $B\widehat\otimes_RR'\simeq B'$.
Pullback identifies each global lift groupoid in
Theorem~\ref{thm:global} with its groupoid of descent data over $R'$,
compatibly with the centraliser equivalence.
\end{proposition}
\begin{proof}
By Lemma~\ref{lem:quantum-factor} and PBW, the complete underlying
module of $B$ is $Q[[\hbar]]$. The invariants defining $Q_s$ are the kernel of
\[
 \Sym_R\mathfrak s\longrightarrow
 \Hom_R(\mathfrak s,\Sym_R\mathfrak s),\qquad
 v\longmapsto(x\mapsto\{x,v\}).
\]
Flat base change preserves this kernel, and finite projectivity
of $\mathfrak s$ identifies the base change of the Hom-module.
Applying the same PBW construction after base change and taking
compatible reductions modulo $\hbar^n$ proves the isomorphism
for $B$. The enveloping relations also commute with base change.

To prove effectiveness, let a local lift over $R'$ carry a descent
arrow $\gamma$ over $R''=R'\otimes_RR'$. Its cocycle condition on
$R'\otimes_RR'\otimes_RR'$ is
$\gamma_{12}+\gamma_{23}=\gamma_{13}$, because splitting arrows
compose by addition. The coefficients of $\gamma$ lie in the
base changes of the $R$-module $\Omega_T^1(A)$, and its constant
coefficient is zero. Exactness of the faithfully flat Amitsur complex
\cite[Tag~023M]{Stacks}, applied coefficientwise, supplies
$c\in\hbar\Omega_{T\otimes_RR'}^1(A'_\hbar)$ with
$\gamma=c_1-c_2$. Changing the local splitting by $c$ replaces
the descent arrow by $\gamma+c_2-c_1=0$.

The new data have equal pullbacks to $R''$ and descend
coefficientwise by faithfully flat module descent
\cite[Tag~023N]{Stacks}. For the connection, applying a local operator to a global section
gives a value with equal pullbacks; its descended value defines the
global operator, and the base Leibniz rule recovers the local one.
The equations descend by faithful flatness. Compatibility modulo
every power of $\hbar$ gives a complete lift, and the same equalizer
condition descends arrows.

For centraliser lifts the same argument uses $\Omega_T^1(Q)$ in
place of $\Omega_T^1(A)$, via the module coordinates for $B$ above.
The iterated-commutator identities bounding differential order may
be checked after faithfully flat base change. Since restriction and
extension in Theorem~\ref{thm:global} commute with pullback, the
equivalence identifies the two descent constructions.
\end{proof}

\subsection{Lifting to every prescribed product}
When the centraliser is commutative, a lift can be constructed by
transporting the classical data. The first step is to identify the
semisimple centres together with their connections.

\begin{lemma}[Comparison of semisimple centres]\label{lem:centre-comparison}
There is a continuous $R[[\hbar]]$-algebra isomorphism
\[
 \Psi:Q_s[[\hbar]]\xrightarrow{\ \sim\ }Z_U
\]
which reduces to the identity and intertwines the connections induced
by every bracket connection on $\mathfrak s$.
\end{lemma}
\begin{proof}
We first construct generators on the complex semisimple model
$\mathfrak s_*$. Set $I=(\Sym_{\mathbb C}\mathfrak s_*)^{\mathfrak s_*}$.
Chevalley's theorem makes $I$ a graded polynomial algebra
\cite[Theorems~10.1 and~10.6]{Chev}. The inner automorphism group acts trivially on $I$.
Consequently the automorphism action on $I$ factors through the finite
quotient
\[
 \Gamma=\operatorname{Aut}(\mathfrak s_*)/
                    \operatorname{Int}(\mathfrak s_*).
\]
This is the automorphism group of the Dynkin diagram, including
permutations of isomorphic simple summands
\cite[Proposition~40.3]{EtingofLie}.

Let $I_+=\bigoplus_{d>0}I_d$ and $W=I_+/(I_+)^2$. The quotient
$W$ records the homogeneous generators modulo products of elements of
smaller positive degree. In each degree choose a linear section
$t_d:W_d\to(I_+)_d$ of the quotient map and replace it by
\[
 t_d^\Gamma=\frac1{|\Gamma|}\sum_{\gamma\in\Gamma}
                 \gamma\,t_d\,\gamma^{-1}.
\]
This remains a section and is $\Gamma$-equivariant. Its image defines
a graded subspace $V\subset I_+$ mapping isomorphically to $W$.
Induction on degree expresses every element of $I$ as a polynomial in
$V$. Moreover, a homogeneous polynomial presentation of $I$ shows
that $V$ has exactly the degrees and multiplicities of its free
generators. Hence the surjection $\Sym_{\mathbb C}V\to I$ is an
isomorphism, by equality of the graded dimensions. Since inner
automorphisms act trivially, $V$ is stable under all automorphisms
of $\mathfrak s_*$.

Write $U_*$ for the completed enveloping algebra of this model, with
relations $xy-yx=\hbar[x,y]$. Send each $v\in V$ to
$\operatorname{sym}(v)$ and extend multiplicatively to
$\Psi_*:I[[\hbar]]\to Z(U_*)$. The images are central because PBW
symmetrisation intertwines the adjoint actions:
\[
 \frac{[x,\operatorname{sym}(p)]}{\hbar}
        =\operatorname{sym}(\{x,p\}),
 \qquad x\in\mathfrak s_*,\quad p\in\Sym\mathfrak s_*.
\]
The same identity identifies $Z(U_*)$ as a complete module with
$I[[\hbar]]$. As a map of complete modules, $\Psi_*$ is the identity modulo
$\hbar$ and hence has a recursively defined inverse. The construction is equivariant for the algebraic automorphism
group. Its action on $I$ factors through $\Gamma$, and PBW is
equivariant in every degree. Both identities persist after scalar
extension, so the local maps agree under the transition
automorphisms and descend to $\Psi$.

Finally, in an \'etale frame any bracket connection has the
form $\nabla_\xi=\xi+\ad_{a_\xi}$, where $\xi$ differentiates the
base coefficients and $a_\xi$ is a local section of $\mathfrak s$.
This follows because the difference from coefficientwise
differentiation is a derivation of the semisimple Lie algebra, hence
inner. On $U$ the corresponding expression is
$\nabla^U_\xi=\xi+I_{a_\xi}$. The inner terms vanish on $Q_s$ and
$Z_U$, respectively, and $\Psi_*$ commutes with coefficientwise
differentiation. Thus $\nabla^U_\xi\Psi=\Psi\nabla_\xi$ in each
frame, and hence globally.
\end{proof}

For the full reductive bundle, form $U_\hbar(\mathfrak g)$ by the
same completed enveloping construction as for $\mathfrak s$. Its
symmetrisation map $\operatorname{sym}_{\mathfrak g}:A[[\hbar]]\to
U_\hbar(\mathfrak g)$ defines the \emph{PBW product}
\[
 f\star_0 g=\operatorname{sym}_{\mathfrak g}^{-1}
   \bigl(\operatorname{sym}_{\mathfrak g}(f)
         \operatorname{sym}_{\mathfrak g}(g)\bigr).
\]
In a Lie algebra frame this is the differential product of
Gutt~\cite{Gutt}, in the symmetrisation form of~\cite[Eq.~(11)]{Dito}
with our parameter convention. Naturality identifies the formulas on overlaps. The central variables
multiply ordinarily for the PBW product. We now determine when
every prescribed product admits a lift. Assume for the theorem and
its corollary that $\Spec R$ is connected and has positive dimension.

\begin{theorem}[The sharp central-rank threshold]\label{thm:threshold}
Every Hamiltonian unfolding on $\mathfrak g^*$ lifts to the PBW product.
For each fixed unfolding $(H,C)$, liftability to every prescribed
relative differential star product is equivalent to
\[
                  \operatorname{rank}\mathfrak z\le1.
\]
\end{theorem}
\begin{proof}
Suppose first that $B$ is commutative.
Lemma~\ref{lem:quantum-factor} gives the flat connection $H_0|_Q$ on
$Q=Q_s\otimes_R\Sym_R\mathfrak z$ and its closed curving $\overline C$.

Let $\Psi:Q_s[[\hbar]]\to Z_U$ be the isomorphism supplied by
Lemma~\ref{lem:centre-comparison}. As in the factorisation proof, projectivity provides an $R$-linear
lift $\ell:\mathfrak z\to B$ of $\mathfrak z\hookrightarrow Q$.
Commutativity of $B$ then extends $\Psi$ and $\ell$
multiplicatively to an algebra map $\Phi:Q[[\hbar]]\to B$,
fitting into
\[
\begin{CD}
 Q_s[[\hbar]] @>{\Psi}>> Z_U\\
 @VVV @VVV\\
 Q[[\hbar]] @>{\Phi}>> B .
\end{CD}
\]
Reduction gives the identity on $Q$, so successive inversion in
the complete, $\hbar$-torsion-free modules makes $\Phi$ an
isomorphism. It transports the classical data to
\[
 D^B_\xi=\Phi\circ H_{0,\xi}\circ\Phi^{-1},\qquad
 C^B=\Phi(\overline C).
\]
These satisfy $F_{D^B}=0$ and
$d_{D^B}C^B=\Phi(d_{H_0}\overline C)=0$.
Since $B$ is commutative, $I_{C^B}=0$ as well. On $Z_U$ the
connection $D^B$ equals $\nabla^U$, by the commutative diagram
and the connection compatibility of $\Psi$. Thus $(D^B,C^B)$
is a centraliser lift, and Theorem~\ref{thm:global} reconstructs
a full lift of the original pair $(H,C)$.

To apply this construction in central rank at most one, we verify
that $B$ is commutative using the factorisation.
If $\mathfrak z=0$, it gives $B=Z_U$.
In rank one, choose a local frame of $\mathfrak z$ and a lift of
its generator to $B$. The symmetrised
expansion in Lemma~\ref{lem:quantum-factor} expresses every element
of $B/\hbar^nB$ as a polynomial in this one lift, with coefficients
in the image of $Z_U$. Those coefficients commute with $B$,
because $B$ centralises the image of $U$. Hence
$B/\hbar^nB$ is commutative for every $n$. This property descends
from the local frames, and passage to the inverse limit makes
$B$ commutative.
For the PBW product in arbitrary central rank, take $j$ to be
the standard inclusion. The central linear generators commute,
and the same expansion gives
\[
 B=\bigl(Z_U\otimes_R\Sym_R\mathfrak z\bigr)^\wedge .
\]
This algebra is commutative, so the construction again applies.

For the converse, suppose $\operatorname{rank}\mathfrak z\ge2$
and fix $m\ge1$.
We construct a product with a commutator varying along the base.
Choose a closed point $p\in\Spec R$.
Smoothness and positive dimension give an element
$f\in\mathfrak m_p$ with nonzero class in
$\mathfrak m_p/\mathfrak m_p^2$ and a tangent vector pairing
nontrivially with that class. The evaluation
$T\to T\otimes_R\kappa(p)$ is surjective, so choose a global
field $\xi\in T$ representing this vector. Then
\[
                    f(p)=0,\qquad \xi(f)(p)\ne0.
\]
Choose linearly independent vectors in the fibre of $\mathfrak z$
at $p$ and dual covectors on their span. Lifting them to global
sections gives $u,v\in\mathfrak z$ and
$\lambda,\eta\in\mathfrak z^\vee$ such that
\[
 g=(\lambda\wedge\eta)(u,v)
   =\lambda(u)\eta(v)-\eta(u)\lambda(v),\qquad g(p)\ne0.
\]
Here finite projectivity identifies the fibre of $\mathfrak z^\vee$
with the dual of the fibre of $\mathfrak z$.

The covectors define relative polynomial derivatives by
$\partial_\lambda(x)=0$ for $x\in\mathfrak s$ and
$\partial_\lambda(z)=\lambda(z)$ for $z\in\mathfrak z$,
and similarly for $\partial_\eta$. These derivatives commute.
They also act by derivations of the PBW product $\star_0$,
whose enveloping relations involve only $\mathfrak s$.
Writing $m_0$ for its multiplication, define
\begin{equation}\label{eq:central-twist}
 m_\star=m_0\circ\exp\!\left[
 \frac{\hbar^{m+1}f}{2}
 \bigl(\partial_\lambda\otimes\partial_\eta
       -\partial_\eta\otimes\partial_\lambda\bigr)\right].
\end{equation}
This is the exponential formula for commuting derivations
\cite[Theorem~2.1]{GZ}. On three factors both parenthesisations
apply the exponential of the sum of the three commuting pairwise
operators and then multiply by the iterated PBW product. Formula~\eqref{eq:central-twist} is unital and
relative differential, and its first-order bracket is the original
linear Poisson bracket. Since $u,v$ are linear central generators,
the formula gives the exact identity
\begin{equation}\label{eq:central-twist-commutator}
                    [u,v]_\star=\hbar^{m+1}fg.
\end{equation}

At $p$ the product specialises to PBW, so $u$ and $v$ become
central. Suppose the chosen base field $\xi$ lifted to a derivation
$D_\xi$ modulo $\hbar^{m+2}$. Apply $D_\xi$ to
\eqref{eq:central-twist-commutator} over $R$ and then evaluate
at $p$. The right-hand side gives
\[
 \hbar^{m+1}\xi(fg)(p)
       =\hbar^{m+1}\xi(f)(p)g(p)\ne0
                              \pmod{\hbar^{m+2}}.
\]
Centrality of the specialised generators gives
\[
 \left([D_\xi u,v]_\star+[u,D_\xi v]_\star\right)\big|_p=0.
\]
This contradiction proves the converse. The same product excludes
a lift through order $m$ for every unfolding on the bundle.
\end{proof}

\begin{corollary}[Failure at a prescribed order]\label{cor:reductive-delayed}
If $\operatorname{rank}\mathfrak z\ge2$, then for every $m\ge1$ there
is a prescribed product to which every classical unfolding lifts
through order $m-1$, while no derivation over a certain $\xi\in T$
exists modulo $\hbar^{m+2}$. Consequently every such unfolding fails
to lift through order $m$.
\end{corollary}
\begin{proof}
The product~\eqref{eq:central-twist} agrees with PBW modulo
$\hbar^{m+1}$. For any full PBW lift the new product Leibniz defect
is therefore $O(\hbar^{m+1})$. The inner-curvature defect is
$O(\hbar^m)$ because the commutator is divided by $\hbar$, and the
Bianchi equation still holds. Taking its coefficients through order
$m-1$, with $D_m$ as the auxiliary coefficient, gives a partial lift
by Definition~\ref{def:partial-lift}. At the next order one needs a
derivation modulo $\hbar^{m+2}$, which the calculation at $p$ excludes.
\end{proof}

\begin{example}[A constant-bundle obstruction]\label{ex:constant-central-twist}
Take $R=\mathbb C[t]$ and the constant bundle
$\mathfrak{sl}_2\oplus\mathbb C u\oplus\mathbb C v$.
With $f=t$ and the derivatives $\partial_u,\partial_v$,
formula~\eqref{eq:central-twist} gives
$[u,v]_\star=\hbar^{m+1}t$. The contradiction at $t=0$ already
applies to the product connection and zero classical curving.
Using the same construction over the central coefficient ring $Z_U$,
one may replace $t$ in the exponential by $t\kappa_\hbar$,
where $\kappa_\hbar\in Z_U$ is the PBW symmetrisation of the
quadratic Casimir. The resulting identity is
$[u,v]_\star=\hbar^{m+1}t\kappa_\hbar$.
Put $\kappa=\kappa_\hbar\bmod\hbar$. At $t=0$, differentiation
along $\partial_t$ gives
$\hbar^{m+1}\kappa\ne0\pmod{\hbar^{m+2}}$, while the commutators
on the derivation-rule side vanish in that fibre.
\end{example}

\begin{remark}[Zero-dimensional bases]
For a zero-dimensional smooth base, $T=0$ and the transport equations
are automatic. The positive-dimensional hypothesis in the threshold
theorem supplies the base direction used in the obstruction.
\end{remark}

For a semisimple bundle, the factorisation gives $A_\hbar\simeq U$
and $B=Z_U$. After normalisation only the central curving remains to
be chosen.

\begin{corollary}[Classification of semisimple lifts]
\label{cor:semisimple-lifts}
If $\mathfrak g=\mathfrak s$, the isomorphism classes of Hamiltonian
lifts form a torsor for
\[
 H^2(T;\hbar Z_U)\simeq\hbar H^2(T;Q_s)[[\hbar]].
\]
Every lift has automorphism group $Z^1(T;\hbar Z_U)$.
\end{corollary}
\begin{proof}
The centraliser connection is $\nabla^U|_{Z_U}$, and
$\Psi(\overline C)$ supplies a closed curving.
Corollary~\ref{cor:central-curvings} therefore gives the torsor and
automorphism groups. Theorem~\ref{thm:global} transfers them to
lifts of the original pair.

For the cohomology identification, connection compatibility of
$\Psi$ gives an isomorphism of cochain complexes
\[
 \bigl(\Omega_T^\bullet(\hbar Z_U),d_{\nabla^U}\bigr)
 \simeq
 \hbar\bigl(\Omega_T^\bullet(Q_s),d_\nabla\bigr)[[\hbar]].
\]
Finite projectivity of $T$ identifies forms coefficientwise, and the
differential acts on each coefficient separately. Taking cohomology
gives the asserted isomorphism.
\end{proof}

\begin{example}[Trace-zero Lie algebras of Azumaya algebras]
Let $\mathcal A$ be an Azumaya $R$-algebra of degree $n\ge2$:
\'etale locally it is the algebra of $n\times n$ matrices
\cite[Tag~0A2J]{Stacks}. Its reduced trace
$\operatorname{trd}:\mathcal A\to R$ is the map which becomes the
ordinary matrix trace in these local presentations. The kernel
\[
 \mathfrak s=\ker(\operatorname{trd}),\qquad [x,y]=xy-yx,
\]
then becomes $\mathfrak{sl}_n$ after the same \'etale base
change. Since $\operatorname{trd}(1)=n$, the map
$a\mapsto a-\operatorname{trd}(a)1/n$ projects onto this kernel.
Thus $\mathfrak s$ is a finite projective semisimple Lie algebra bundle, and
the corollary applies to every Hamiltonian unfolding on its dual
and every prescribed relative differential star product.
\end{example}

\subsection{A nontrivial centre line}
We illustrate the rank-one criterion on a nontrivial centre line.
The semisimple summand also has monodromy. Two local functions that
change sign together will give a global central curving.

Let $E$ be the elliptic curve $y^2=x(x-1)(x+1)$, with origin $O$
at infinity and $P=(0,0)$. Set $S_0=E\setminus\{O\}$ and
$M=\mathcal O_E(P-O)|_{S_0}$. The identity
$\operatorname{div}(x)=2P-2O$ gives $M^{\otimes2}\simeq\mathcal O_{S_0}$.
The line $M$ is nontrivial: a trivialisation would make $P-O$
linearly equivalent on $E$ to a divisor supported at $O$.
Degree forces that divisor to be zero, so a function with divisor
$P-O$ would give a degree-one map $E\to\mathbb P^1$, contradicting
the genus of $E$.

Choose a trivialisation of $M^{\otimes2}$ and use it as the
multiplication on the summand $M$ in
$\mathcal O_{S_0}\oplus M$. This defines a connected \'etale
double cover
$\widetilde S_0=\operatorname{Spec}_{S_0}(\mathcal O_{S_0}\oplus M)$.
In a local frame of $M$, its equation is $t^2=u$ with $u$ a unit,
which proves \'etaleness. A disconnected double cover would split
as two copies of $S_0$ and trivialise $M$, so this cover is connected.
Over $S=S_0\times\mathbb G_m^2$, descend the trivial Lie algebra
bundle on $\widetilde S=\widetilde S_0\times\mathbb G_m^2$
by combining the deck involution with
\[
 \mathfrak{sl}_3\oplus\mathbb Cw
   \longrightarrow\mathfrak{sl}_3\oplus\mathbb Cw,\qquad
 (X,w)\longmapsto(-X^{\mathsf t},-w).
\]
The centre line of the descended bundle is the pullback of $M$;
restriction to $S_0\times\{(1,1)\}$ shows that it remains
nontrivial. The involution on $\mathfrak{sl}_3$ gives its outer
monodromy.

Let $a,b$ be the torus coordinates and
$\omega=d\log a\wedge d\log b$. Coefficientwise differentiation in
the trivial bundle upstairs commutes with the involution and
descends to a flat Poisson connection $H_{\mathrm{flat}}$.
Using the trace pairing to identify $\mathfrak{sl}_3$ with its
dual, put $p_3(X)=\operatorname{tr}(X^3)$.
Both $p_3$ and $w$ change sign under the involution. Thus
\[
                         c=p_3w\,\omega
\]
is a global central curving, closed for $H_{\mathrm{flat}}$.
It represents a nonzero class: on a torus fibre over $S_0$,
the coefficient of $a^0b^0$ in an exact logarithmic two-form
is zero, whereas that coefficient in $c$ is the nonzero
polynomial $p_3w$.

We can also make the semisimple curvature visible. The matrix
$K=E_{12}-E_{21}$ satisfies $-K^{\mathsf t}=K$, so it defines a
global section of the semisimple summand. Change the flat splitting
by the global one-form $\lambda=aK\,d\log b$. Since
$d_{H_{\mathrm{flat}}}\lambda=aK\omega$ and
$[\lambda,\lambda]=0$, formula~\eqref{eq:change} gives
\[
 H=H_{\mathrm{flat}}+X_\lambda,\qquad
 C=(p_3w+aK)\omega .
\]
The induced bracket connection is
$\nabla=H_{\mathrm{flat}}|_{\mathfrak s}+\ad_\lambda$.
The displayed pair is already normalised relative to this $\nabla$.
Its curvature is represented by $\Omega=aK\omega$;
on the two Euler fields $a\partial_a,b\partial_b$, the
Hamiltonian curvature is $X_{aK}$. The residual curving
$C-\Omega$ is precisely $c$.

For the PBW product, coefficientwise differentiation lifts the flat
connection. The form $\operatorname{sym}(p_3)w\,\omega$ lifts $c$.
The quantum change of splitting by $\lambda$ then gives a lift
of the displayed curved pair $(H,C)$.
Theorem~\ref{thm:threshold} extends this existence statement to
every prescribed relative differential product on the bundle.

\Needspace{10\baselineskip}

\section{The coupling complex and lifting obstructions}\label{sec:quantum}

Section~\ref{sec:reductive} determined when a Hamiltonian unfolding
on the dual of a reductive Lie algebra bundle lifts to every
prescribed star product. We now return to a general smooth affine
Poisson family and fix both an unfolding $(H,C)$ and a relative
differential star product $\star$. For this particular product,
can we correct the connection $H$ and its curvature Hamiltonians $C$
so that they satisfy the quantum transport equations? Since the
product is a formal series in $\hbar$, we can approach this question
one order at a time.

At first order, the product itself supplies the starting point.
After normalisation, its divided commutator has the expansion
\[
  \frac{[f,g]_\star}{\hbar}
  = \{f,g\}+\hbar\,\beta(f,g)+O(\hbar^2),
\]
where $\beta$ is a relative Poisson two-cocycle. Thus $\beta$
records the first change in the bracket of fibre observables.
To describe transport, we must also allow directions along the base.
With one base direction $\xi$ and one observable $f$, the new datum
is the correction to $H_\xi(f)$. With two base directions $\xi,\eta$,
it is the correction to the curvature Hamiltonian $C_{\xi\eta}$.
The coupling complex brings these data together as the components
of a single two-cochain. We prove that a first-order lift exists
exactly when the two correction terms can be chosen to make this
cochain closed.

\medskip
\noindent
In Subsection~4.1 we extract the class of $\beta$ from the star
product and determine when a classical Poisson derivation admits
a first correction. Subsection~4.2 constructs the coupling Lie
algebroid and identifies the components of its cochain differential
with the lifting equations. In Subsection~4.3 this gives three
successive tests for a first-order lift, corresponding to the
Leibniz rule, the curvature equation and the Bianchi identity.
Subsection~4.4 explains how the chosen curvature Hamiltonians enter
the coupling bracket and gives an example in which they obstruct a
lift even though the connection already preserves the star product.

Subsection~4.5 passes to higher orders. For each specified partial
lift, we construct a cohomology class whose vanishing is equivalent
to extending that lift by one more order. Subsection~4.6 asks
whether this can be tested after a change of base. We prove that
a finite \'etale cover detects the obstruction, and we give an
example in which full lifts exist on a Zariski open cover although
a first-order lift is already obstructed globally.

\subsection{Differential corrections and the quantum Poisson class}
To extract the first constraint from the product, use the normalisation
$B_1=\pi/2$ of Lemma~\ref{lem:sym} throughout this section. Write
$\Alt B(f,g)=B(f,g)-B(g,f)$.

\begin{lemma}[The first quantum Poisson coefficient]\label{lem:quantum-poisson-class}
For a prescribed product with $B_1=\pi/2$,
\[
 \mu=\pi+\hbar\beta+O(\hbar^2),\qquad \beta=\Alt B_2,
\]
where $\beta$ is a relative Poisson two-cocycle. Two such prescribed products are differentially equivalent modulo
$\hbar^3$ exactly when their classes $[\beta]\in H^2_\pi(A/R)$ agree.
\end{lemma}
\begin{proof}
The divided commutator is a derivation of the associative product in
each argument. Its coefficient of $\hbar$ gives
\begin{align*}
 \beta(a,bc)-\beta(a,b)c-b\beta(a,c)
 &=B_1(\pi(a,b),c)+B_1(b,\pi(a,c))\\
 &\quad-\pi(a,B_1(b,c)).
\end{align*}
The right-hand side is zero by $B_1=\pi/2$ and Poisson Jacobi.
Thus $\beta(a,-)$ is a relative derivation; skew-symmetry gives the
same assertion in the other argument. The coefficient of $\hbar$
in Jacobi for $\mu$ is $[\pi,\beta]=0$.

An equivalence $\id+\hbar W+O(\hbar^2)$ between normalised
presentations has $bW=0$, so $W$ is a relative vector field.
Transporting $\mu$ gives
\[
 \beta'=\beta+\cL_W\pi=\beta-d_\pi W.
\]
Conversely, if two normalised presentations have the same class,
conjugation of one product by $\exp(\hbar W)$, for a suitable relative
vector field $W$, matches their alternating second coefficients.
Their second coefficients then differ by a symmetric differential
Hochschild cocycle: the order-two associativity equations have the
same first coefficient. Lemma~\ref{lem:sym} removes this difference
by an equivalence $\id+\hbar^2Q$ modulo $\hbar^3$.
\end{proof}

For a first-order lift, write
\[
 D_\xi=H_\xi+\hbar V_\xi+\hbar^2U_\xi,\qquad C_\hbar=C+\hbar c.
\]
In Definition~\ref{def:partial-lift}, the recorded coefficients are the
\emph{connection correction} $V$ and the \emph{curving correction}
$c\in\Omega_T^2(A)$. The \emph{auxiliary coefficient} $U$ is used to
satisfy the product Leibniz rule modulo $\hbar^3$. The curvature and
Bianchi equations require
\begin{equation}\label{eq:innerlift}
 F_D(\xi,\eta)=I_{C_\hbar(\xi,\eta)}\pmod{\hbar^2},
 \qquad d_D C_\hbar=0\pmod{\hbar^2}.
\end{equation}
Consider first the product equation for one projectable Poisson
field $Y$. Since $\cL_Y$ commutes with $d_\pi$, it acts on relative
cohomology by $\cL_Y[\gamma]=[\cL_Y\gamma]$. For the lift of a
Poisson field to the product obtained by formality from the given
Poisson tensor, see~\cite[Proposition~3.1]{Sharygin}.

\begin{proposition}[Lifting a Poisson derivation]\label{thm:symbol}
A projectable Poisson field $Y$ admits a differential lift
$Y+\hbar V+\hbar^2U$ modulo $\hbar^3$ if and only if
$\cL_Y[\beta]=0$. Its first connection corrections are precisely
the relative vector fields satisfying
\begin{equation}\label{eq:first}
 \cL_Y\beta+\cL_V\pi=0.
\end{equation}
For each such $V$, the auxiliary coefficients $U$ are the solutions of
\begin{equation}\label{eq:hoch}
 bU=\cL_YB_2+\cL_VB_1.
\end{equation}
They form an affine space over $\Der_R A$ and are relative differential
operators.
\end{proposition}
\begin{proof}
Insert $D=Y+\hbar V+\hbar^2U$ into
$D(f\star g)-D(f)\star g-f\star D(g)$. The coefficient of
$\hbar$ is $\cL_YB_1-bV$. Since $Y$ preserves $\pi$ and
$B_1=\pi/2$, its vanishing says that $V$ is a derivation.
The coefficient of $\hbar^2$ is
$\cL_YB_2+\cL_VB_1-bU$, which gives~\eqref{eq:hoch}.

Set $E=\cL_YB_2+\cL_VB_1$. The order-two associativity identity
expresses $bB_2$ as a quadratic expression in $B_1$.
Applying $\cL_Y$ annihilates that expression because
$\cL_YB_1=0$. Also $\cL_VB_1$ is a biderivation and hence a
Hochschild cocycle. It follows that $bE=0$.
The cochain $E$ is normalised and bidifferential, and
\[
 \Alt E=\cL_Y\beta+\cL_V\pi.
\]
Lemma~\ref{lem:sym} therefore supplies a normalised relative operator
$U$ exactly when~\eqref{eq:first} holds. The same lemma makes every
solution differential. Two solutions differ by a normalised
Hochschild one-cocycle, hence by a relative derivation.
Finally, $d_\pi V=-\cL_V\pi$ identifies~\eqref{eq:first} with
$d_\pi V=\cL_Y\beta$. This proves the cohomological criterion.
\end{proof}

For a connection the corrections must also depend $R$-linearly on
the base field. This is where projectivity of $T$ enters.

\begin{lemma}[Simultaneous differential corrections]\label{lem:simultaneous-corrections}
There are $R$-linear corrections $V,U$ making
$H+\hbar V+\hbar^2U$ a derivation connection modulo $\hbar^3$
if and only if $\cL_{H_\xi}[\beta]=0$ for every $\xi\in T$.
Every $R$-linear $V$ satisfying~\eqref{eq:first} admits such an $R$-linear
auxiliary coefficient $U$.
\end{lemma}
\begin{proof}
The map
$\xi\mapsto\cL_{H_\xi}\beta$ is $R$-linear and takes values in
relative Poisson cocycles. Under the stated hypothesis its values lie
in $\im(d_\pi:\Der_R A\to C^2_\pi(A/R))$.
Projectivity of $T$ lifts this map through the indicated surjection,
giving an $R$-linear $V$ with
$d_\pi V_\xi=\cL_{H_\xi}\beta$.

For a prescribed such $V$, let
$E_\xi=\cL_{H_\xi}B_2+\cL_{V_\xi}B_1$.
These cochains are symmetric by~\eqref{eq:first} and depend
$R$-linearly on $\xi$. The module of pairs
\[
 \{(\xi,U):\ U\in\Hom_R(A,A),\ U(1)=0,\ bU=E_\xi\}
\]
surjects onto $T$ by Lemma~\ref{lem:sym}. A section gives the required
$R$-linear choice, and every chosen $U_\xi$ is differential by that
lemma. Since $T$ is finitely generated, choosing finitely many module
generators also gives a common differential-order bound for these
coefficients. Conversely, the coefficients of the derivation equation imply
\eqref{eq:first}, hence the stated cohomological vanishing.
\end{proof}

\subsection{The coupling Lie algebroid and its cochains}
The obstruction to lifting a single Poisson field lies in relative
Poisson cohomology. To include curvature we also need horizontal
arguments. Recall the extension $E=A\oplus T$ over $R$ from
Proposition~\ref{thm:reconstruction}. Applying the relative
differential $\dd:A\to\Omega^1_{A/R}$ to the function component
gives the $R$-linear comparison map
\[
 \chi:E\longrightarrow L:=\Omega^1_{A/R}\oplus(A\otimes_R T),
 \qquad \chi(f,\xi)=\dd f+\widehat\xi,\qquad
 \widehat\xi=(0,1\otimes\xi).
\]

Define the $A$-linear anchor $a_L:L\to\Der_k A$ by
$a_L(d_{\mathrm v}f)=X_f$ and $a_L(\widehat\xi)=H_\xi$.
Its restriction to $\Omega^1_{A/R}$ is the Poisson anchor
$a_\pi:\Omega^1_{A/R}\to\Der_R A$. The splitting formulas for
coupling Dirac structures~\cite[Propositions~2.4 and~2.10]{BF}
suggest the following brackets on generators.
\begin{align*}
 [d_{\mathrm v}f,d_{\mathrm v}g]&=d_{\mathrm v}\{f,g\},\\
 [\widehat\xi,d_{\mathrm v}f]&=d_{\mathrm v}(H_\xi f),\\
 [\widehat\xi,\widehat\eta]&=\widehat{[\xi,\eta]}+d_{\mathrm v}C_{\xi\eta}.
\end{align*}
Thus $L$ records the relative differential of the curvature
Hamiltonian. Extend these brackets by skew-symmetry and the Leibniz rule
$[\ell,a\ell']=a[\ell,\ell']+a_L(\ell)(a)\ell'$.

\begin{lemma}[The coupling bracket]\label{lem:coupling-bracket}
These formulas make $L$ a Lie--Rinehart algebra over $A$, acting on
$A$ through $a_L$. Under a change of splitting~\eqref{eq:change},
the map
\[
 L_{H',C'}\longrightarrow L_{H,C},\qquad
 d_{\mathrm v}f\longmapsto d_{\mathrm v}f,\quad
 \widehat\xi'\longmapsto\widehat\xi+d_{\mathrm v}b_\xi
\]
is an isomorphism of Lie--Rinehart algebras.
\end{lemma}
\begin{proof}
The vertical bracket is the cotangent Poisson bracket
\cite[Theorem~3.8]{Huebschmann}. The mixed
bracket is the Lie derivative of relative forms by the projectable
field $H_\xi$. These descriptions respect the relations defining
$\Omega^1_{A/R}$. The connection law for $H$ and $R$-bilinearity of
$C$ give the corresponding compatibility with relations among the
horizontal generators.

Anchor preservation on two vertical generators follows from
$[X_f,X_g]=X_{\{f,g\}}$. On a mixed pair it is the identity
$[H_\xi,X_f]=X_{H_\xi f}$. On two horizontal generators it is
$[H_\xi,H_\eta]=H_{[\xi,\eta]}+X_{C_{\xi\eta}}$.
With anchor preservation established, the Jacobiator is
$A$-multilinear, so it suffices to evaluate it on the generators.
For three vertical generators it is Poisson Jacobi. For one horizontal
and two vertical generators it is the relative differential of the
identity expressing $\cL_{H_\xi}\pi=0$. For two horizontals and
$d_{\mathrm v}f$, it is $d_{\mathrm v}((F_H-X_C)_{\xi\eta}f)=0$.
For three horizontals, the vertical component is
$d_{\mathrm v}(d_HC)_{\xi\eta\zeta}=0$, and the horizontal component is
Jacobi in $T$. Thus the bracket is a Lie bracket.

For the change-of-splitting map, anchor preservation is
$H'_\xi=H_\xi+X_{b_\xi}$. Expanding the bracket of
$\widehat\xi+d_{\mathrm v}b_\xi$ and $\widehat\eta+d_{\mathrm v}b_\eta$ gives
\[
 \widehat{[\xi,\eta]}+
 d_{\mathrm v}\bigl(C_{\xi\eta}+H_\xi b_\eta-H_\eta b_\xi+
                   \{b_\xi,b_\eta\}\bigr).
\]
The transformation formula for $C'$ identifies this with the image
of $\widehat{[\xi,\eta]}' + d_{\mathrm v}C'_{\xi\eta}$.
The mixed and vertical brackets are preserved by the same generator
formulas. Changing $b$ to $-b$ gives the inverse.
\end{proof}

\begin{definition}[Coupling Lie algebroid and coupling complex]\label{def:coupling-complex}
The Lie--Rinehart algebra of Lemma~\ref{lem:coupling-bracket} is the
\emph{coupling Lie algebroid} of $(H,C)$ over $A$. Its \emph{coupling complex} is
\[
 C^n(L;A)=\Hom_A(\wedge_A^nL,A)
\]
with the Chevalley--Eilenberg differential induced by the bracket
and anchor~\cite[Section~4]{Rinehart}. Write $H^n(L;A)$ for its
cohomology.
The \emph{base-form filtration} is
\[
 F^pC^n(L;A)=\bigoplus_{i\geq p}\mathscr C^{i,n-i},
 \qquad
 \mathscr C^{p,q}=\Omega_T^p(C^q_\pi(A/R)),
\]
where the direct-sum decomposition uses the displayed splitting of
$L$. We write $F^1$ for its first filtered subcomplex, and $H^n(F^1)$
and $Z^n(F^1)$ for its cohomology and cocycles.
\end{definition}

In this decomposition, horizontal arguments are placed before vertical
arguments. The differential convention is
\begin{align*}
 (d_LB)(\ell_0,\ldots,\ell_n)
 &=\sum_i(-1)^i a_L(\ell_i)
 B(\ell_0,\ldots,\widehat\ell_i,\ldots,\ell_n)\\
 &\quad+\sum_{i<j}(-1)^{i+j}
 B([\ell_i,\ell_j],\ell_0,\ldots,\widehat\ell_i,\ldots,
                                      \widehat\ell_j,\ldots,\ell_n).
\end{align*}
Thus the cochains are $A$-multilinear and the differential is
$k$-linear. The three kinds of brackets above make $d_L$ a sum of
operators of bidegrees $(0,1)$, $(1,0)$ and $(2,-1)$; in particular,
it preserves $F^p$. The first summand is $(-1)^p d_\pi$ on
$\mathscr C^{p,q}$. The decomposition and filtration agree with
the extension-complex construction of~\cite[Proposition~3.2]{Brahic}
and the bigraded framework of~\cite[Section~2]{PVV}.

Restriction to vertical arguments is a map of complexes. Extending
vertical cochains by zero on horizontal arguments gives a section
of the underlying graded modules, and hence the short exact sequence
\begin{equation}\label{eq:restriction}
 0\longrightarrow F^1\longrightarrow C^\bullet(L;A)
 \xrightarrow{\ r\ }C^\bullet_\pi(A/R)\longrightarrow0.
\end{equation}
In particular, a degree-two cochain extending $\beta$ is specified
by a pair $(V,c)$ through
\[
 B(d_{\mathrm v}f,d_{\mathrm v}g)=\beta(f,g),\qquad
 B(\widehat\xi,d_{\mathrm v}f)=V_\xi(f),\qquad
 B(\widehat\xi,\widehat\eta)=c_{\xi\eta}.
\]
Here $V\in\Omega_T^1(\Der_R A)$ and $c\in\Omega_T^2(A)$.
Define
\[
 K_{\xi\eta}=(d_HV)_{\xi\eta}-\beta(C_{\xi\eta},-),
 \qquad J=d_Hc+V\cdot C,
\]
where
\begin{align*}
 (d_HV)_{\xi\eta}
 &=[H_\xi,V_\eta]-[H_\eta,V_\xi]-V_{[\xi,\eta]},\\
 (V\cdot C)_{\xi\eta\zeta}
 &=V_\xi(C_{\eta\zeta})+V_\eta(C_{\zeta\xi})+V_\zeta(C_{\xi\eta}).
\end{align*}

\begin{lemma}[The four components of the differential]\label{lem:coupling-components}
For $B=(\beta,V,c)$, the components of $d_LB$ in increasing
horizontal degree are
\begin{equation}\label{eq:four}
 d_\pi\beta,\qquad
 (\cL_{H_\xi}\beta+\cL_{V_\xi}\pi)_\xi,\qquad
 K-X_c,\qquad J.
\end{equation}
For the coefficient $\beta=\Alt B_2$ of the prescribed product,
closedness of $B$ is equivalent to the existence of a first-order
lift with corrections $(V,c)$.
\end{lemma}
\begin{proof}
On three vertical entries, the differential is the relative Poisson
differential. On one horizontal and two vertical entries, the defining
formula gives
\begin{align*}
 (d_LB)(\widehat\xi,d_{\mathrm v}f,d_{\mathrm v}g)
 &=H_\xi\beta(f,g)-\beta(H_\xi f,g)-\beta(f,H_\xi g)\\
 &\quad+V_\xi\{f,g\}-\{V_\xi f,g\}-\{f,V_\xi g\}\\
 &=(\cL_{H_\xi}\beta+\cL_{V_\xi}\pi)(f,g).
\end{align*}
On two horizontal and one vertical entry it gives
\begin{align*}
 (d_LB)(\widehat\xi,\widehat\eta,d_{\mathrm v}f)
 &=H_\xi(V_\eta f)-H_\eta(V_\xi f)-V_{[\xi,\eta]}f\\
 &\quad-V_\eta(H_\xi f)+V_\xi(H_\eta f)
       -\beta(C_{\xi\eta},f)+\{f,c_{\xi\eta}\}\\
 &=(K_{\xi\eta}-X_{c_{\xi\eta}})(f).
\end{align*}
On three horizontal entries, the horizontal parts of their brackets
produce $d_Hc$. Their vertical parts produce
$\sum_{\rm cyc}V_\xi(C_{\eta\zeta})$. This proves~\eqref{eq:four}.

By Lemma~\ref{lem:simultaneous-corrections}, vanishing of the
one-horizontal component is precisely the condition for an
$R$-linear auxiliary $U$ satisfying the product derivation rule.
Expanding the remaining two transport equations gives
\begin{align*}
 F_D-I_{C+\hbar c}
 &=\hbar\bigl(d_HV-\beta(C,-)-X_c\bigr)+O(\hbar^2),\\
 d_D(C+\hbar c)
 &=\hbar(d_Hc+V\cdot C)+O(\hbar^2).
\end{align*}
Their coefficients are the final two components in~\eqref{eq:four}.
Together with $d_\pi\beta=0$, this proves the equivalence.
\end{proof}

Changing the correction pair amounts to adding a cochain
$(W,c)\in (F^1)^2=\Omega_T^1(\Der_R A)\oplus\Omega_T^2(A)$.
By~\eqref{eq:four}, the three defects then change by
\begin{equation}\label{eq:FilDiff}
 d_L(W,c)=(\cL_W\pi,\ d_HW-X_c,\ d_Hc+W\cdot C).
\end{equation}
\subsection{The first-order criterion and its three tests}
We can now test for a lift before choosing its corrections.
The connecting map of~\eqref{eq:restriction} measures whether a
vertical cocycle extends to a closed coupling cochain. Write
$Z^2_\pi(A/R)$ for relative Poisson two-cocycles and extend $\beta$
by zero on horizontals to obtain $\widetilde\beta$. Then
\[
\begin{CD}
 Z^2_\pi(A/R) @>{\beta\mapsto[d_L\widetilde\beta]}>> H^3(F^1)\\
 @VVV @|\\
 H^2_\pi(A/R) @>{\delta}>> H^3(F^1).
\end{CD}
\]
Any two extensions of $\beta$ differ by a cochain of $F^1$, so their
differentials give the same class.

The relevant part of the long
exact sequence is
\[
 H^2(L;A)\xrightarrow{r_*}H^2_\pi(A/R)
          \xrightarrow{\delta}H^3(F^1).
\]

\begin{theorem}[The first-order lifting criterion]\label{thm:main}
The Hamiltonian unfolding $(H,C)$ admits a first-order lift to the
prescribed product if and only if $\delta[\beta]=0$, equivalently
\begin{equation}\label{eq:edge}
 [\beta]\in\operatorname{im}\!\left(
 H^2(L;A)\longrightarrow H^2_\pi(A/R)\right).
\end{equation}
Every closed extension of $\beta$ is realised by a differential lift.
The criterion is invariant under changes of splitting and relative
differential product equivalence.
\end{theorem}
\begin{proof}
By Lemma~\ref{lem:coupling-components}, closed extensions
$(\beta,V,c)$ correspond exactly to first-order corrections $(V,c)$
admitting an auxiliary $U$. To pass from cocycles to cohomology,
suppose a closed cochain restricts to $\beta+d_\pi Y$. Extend $Y$
by zero on horizontals and subtract its coupling differential.
The resulting closed cochain restricts to $\beta$ itself. Thus a
lift exists precisely under~\eqref{eq:edge}. The cohomology
sequence of~\eqref{eq:restriction} identifies this image with
$\ker\delta$.

The isomorphism in Lemma~\ref{lem:coupling-bracket} preserves
vertical restriction, so it identifies the criterion for two
classical splittings. Lemma~\ref{lem:quantum-poisson-class} gives
invariance of $[\beta]$ under product equivalence; conjugation
transports the resulting differential lifts.
\end{proof}

To compute the obstruction by horizontal degree, we need the
connection on vertical cohomology. Its action
\[
 \nabla^H_\xi[\gamma]=[\cL_{H_\xi}\gamma]
 \quad\hbox{on }H^q_\pi(A/R)
\]
is well defined because $\cL_{H_\xi}$ commutes with $d_\pi$.
Its curvature is induced by $\cL_{X_{C_{\xi\eta}}}$, which is
cochain homotopic to zero through the identity
\[
 \cL_{X_f}=-(d_\pi\ad_f+\ad_f d_\pi),\qquad
 \ad_f(\gamma)=[f,\gamma].
\]
Thus $\nabla^H$ is flat. The same homotopy shows that it is
independent of the classical splitting. In particular,
$H^0_\pi(A/R)=Z_\pi$, and $H^1_\pi(A/R)$ consists of relative
Poisson vector fields modulo Hamiltonian vector fields.
In the tests below, a bar denotes this \emph{outer class},
also for differential forms with Poisson-vector-field coefficients.

Starting from $\beta$, solve the components of~\eqref{eq:four}
in order of increasing horizontal degree.

\begin{enumerate}
\item \emph{Multiplicativity.} Test $\nabla^H[\beta]=0$. If it
vanishes, choose $V\in\Omega_T^1(\Der_R A)$ with
$\cL_H\beta+\cL_V\pi=0$.
\item \emph{Hamiltonian curvature.} Form $K=d_HV-\beta(C,-)$ and
test $[\overline K]=0$ in $H^2(T;H^1_\pi(A/R))$. Vanishing allows
us to adjust $V$ and choose $c$ with $K=X_c$.
\item \emph{Bianchi.} For these choices $J=d_Hc+V\cdot C$ is a
closed central three-form. The last test asks whether its class vanishes in
\begin{equation}\label{eq:quotient}
 H^3(T;Z_\pi)\big/\operatorname{im}\!\left(
 \tau:H^1(T;H^1_\pi(A/R))\longrightarrow H^3(T;Z_\pi)
 \right).
\end{equation}
The map $\tau$ measures the freedom left by the curvature equation.
Represent an outer one-cocycle by a Poisson-vector-field-valued
one-form $W$ and choose $b\in\Omega_T^2(A)$ with $d_HW=X_b$. Then
$\tau([\overline W])=[d_Hb+W\cdot C]$.
\end{enumerate}

These tests are the outgoing differentials from bidegree $(0,2)$ in
the spectral sequence of the base-form filtration
\cite[Section~3.2]{Brahic}, \cite[Sections~2 and~4]{PVV}.
Projectivity of $\wedge_R^pT$ and the vertical differential
$(-1)^p d_\pi$ give
\[
 E_1^{p,q}=\Omega_T^p(H^q_\pi(A/R)),\qquad
 d_1=d_{\nabla^H},\qquad
 E_2^{p,q}=H^p(T;H^q_\pi(A/R)).
\]
A class that survives all the differentials is called a
\emph{permanent cycle}.

\begin{proposition}[Successive tests for the first correction]\label{prop:first-order-tests}
The three tests above represent $d_1[\beta]$, $d_2[\beta]$ and
$d_3[\beta]$, respectively. Their successive vanishing is equivalent
to a first-order Hamiltonian lift. Equivalently, $[\beta]$ is a
permanent cycle in bidegree $(0,2)$.
\end{proposition}
\begin{proof}
The one-horizontal component of~\eqref{eq:four} represents
$d_1[\beta]$ in vertical cohomology. When it vanishes, choose $V$
by Lemma~\ref{lem:simultaneous-corrections}. For $B=(\beta,V,0)$,
the first possible component of $d_LB$ is then $K$, in horizontal
degree two. The corresponding components of $d_L^2B=0$ give
$d_\pi K=0$ and vertical exactness of $d_HK$. Hence $K$ has Poisson
values, $\overline K$ is closed, and $[\overline K]=d_2[\beta]$.

If this class vanishes, choose an outer one-form
$\overline W$ with $d_H\overline W=-\overline K$.
Projectivity of $T$ lifts it to a one-form $W$ of Poisson vector
fields. Replacing $V$ by $V+W$ preserves the first equation
and makes $K$ Hamiltonian-valued. Projectivity of
$\wedge_R^2T$ then lifts $K$ through the surjection
$A\to\{\hbox{Hamiltonian vector fields}\}$, giving $K=X_c$.
For this choice $d_L(\beta,V,c)=J$ has horizontal degree three.
Applying $d_L$ once more shows that $J$ is central and
$d_H$-closed.

We next determine its indeterminacy. Two choices solving the first
two equations differ by a pair $(W,b)$ with Poisson-valued $W$
and $d_HW=X_b$. Formula~\eqref{eq:FilDiff} gives
\[
 d_L(0,W,b)=(0,0,0,d_Hb+W\cdot C).
\]
Thus the resulting change of $[J]$ is the stated $\tau$.
Changing $b$ by a central two-form changes its representative by
an exact central three-form. Changing $W$ by a Hamiltonian-valued
one-form $X_\alpha$ is accompanied by $b\mapsto b+d_H\alpha$;
the change of the displayed differential is zero because
$(0,X_\alpha,d_H\alpha)=d_L\alpha$.
Finally, an outer coboundary has a representative
$W=d_HY+X_\alpha$ with $Y$ a relative Poisson vector field.
Taking $b=d_H\alpha-Y(C)$ gives
\[
 (0,W,b)=d_L(Y+\alpha),
\]
so its transgression is zero. Projectivity of $T$ allows the required cochain lifts. Thus $\tau$
descends to $H^1(T;H^1_\pi(A/R))$.

Conversely, every such pair $(W,b)$ can be added to a solution
$(V,c)$ of the first two equations. Hence $\im\tau$ is exactly
the indeterminacy of $[J]$. If its image in~\eqref{eq:quotient}
vanishes, choose $(W,b)$ making $J$ exact, and subtract a central
primitive from $c$. All four components of~\eqref{eq:four} then
vanish, yielding a lift by Theorem~\ref{thm:main}.

In spectral-sequence notation, $\tau=d_2^{1,1}$, so the quotient
\eqref{eq:quotient} is $E_3^{3,0}$ and the final test is
$d_3[\beta]$. The bidegree $(r,1-r)$ of $d_r$ excludes further outgoing
differentials and all incoming ones at $(0,2)$. The filtration is finite in each total degree,
so the spectral sequence converges to coupling cohomology
\cite[Tag~012W]{Stacks}. The image of vertical restriction is
therefore precisely the permanent-cycle subspace of $E_1^{0,2}$.
This gives the last equivalence.
\end{proof}

\begin{corollary}[A one-dimensional base]\label{cor:curve-first-order}
If $\operatorname{rank}_R T=1$, the unfolding $(H,C)$ admits a
first-order lift if and only if $\nabla^H[\beta]=0$.
\end{corollary}
\begin{proof}
The modules $\wedge_R^pT$ vanish for $p\geq2$.
Thus the curvature and Bianchi components of~\eqref{eq:four}
are zero, and the multiplicativity equation is the single test.
\end{proof}

\subsection{Curvings and the coupling bracket}
For a fixed connection $H$, the coupling bracket records the fibre
differentials of its curvature Hamiltonians. We identify which
changes of curving it detects, then exhibit a curving that obstructs
a lift even though $H$ already acts by quantum derivations.

\begin{definition}[Relative constants and curvings with fixed coupling data]
The algebra of \emph{relative constants} is
\[
 K_{A/R}=\ker\bigl(\dd:A\longrightarrow\Omega^1_{A/R}\bigr).
\]
For a fixed connection $H$ and a fixed coupling bracket on the
displayed module $L=\Omega^1_{A/R}\oplus(A\otimes_R T)$, let
$\mathcal C_{H,L}$ be the groupoid whose objects are closed curvings
that induce this connection and bracket. Its arrows are changes of
splitting that induce the identity on the displayed algebroid $L$.
\end{definition}

\begin{proposition}[Curving information]\label{prop:curving-information}
The restriction of $H$ makes $K_{A/R}\subseteq Z_\pi$ a flat
$T$-module. For two closed curvings $C,C'$ of $H$,
\[
 L_{H,C'}=L_{H,C}
 \quad\Longleftrightarrow\quad
 C'-C\in\Omega_T^2(K_{A/R}).
\]
When $\mathcal C_{H,L}$ is nonempty, its isomorphism classes form a
torsor for $H^2(T;K_{A/R})$, and each object has automorphism group
$Z^1(T;K_{A/R})$.
\end{proposition}
\begin{proof}
The anchors agree because the connection is fixed. The only bracket
that involves the curving is the bracket of two horizontal generators,
and its change is
\[
 [\widehat\xi,\widehat\eta]_{C'}
 -[\widehat\xi,\widehat\eta]_C
       =\dd(C'_{\xi\eta}-C_{\xi\eta}).
\]
Thus equality of the displayed brackets is equivalent to the stated
condition on $C'-C$.

For $f\in K_{A/R}$, the identity $X_f=a_\pi(\dd f)$ gives $X_f=0$;
hence $K_{A/R}\subseteq Z_\pi$. Moreover,
\[
 \dd(H_\xi f)=\mathcal L_{H_\xi}(\dd f)=0,
 \qquad
 F_H(\xi,\eta)f=\{C_{\xi\eta},f\}=0.
\]
These formulas prove preservation and flatness.

Fix an object $C$ of $\mathcal C_{H,L}$. The others are exactly
$C+z$ for $z\in Z^2(T;K_{A/R})$, by the bracket and closure
conditions. A splitting change sends $\widehat\xi$ to
$\widehat\xi+\dd b_\xi$, so it fixes the displayed algebroid exactly
for $b\in\Omega_T^1(K_{A/R})$. It then adds $d_Hb$ to the curving
by~\eqref{eq:change}. Taking orbits gives the $H^2$-torsor, and
taking stabilisers gives $Z^1(T;K_{A/R})$.
\end{proof}

In the resulting sequence
\[
 C\xmapsto{\ \dd\ }\dd C
     \xmapsto{\ a_\pi\ }X_C,
\]
the coupling bracket retains $\dd C$, and the connection curvature
retains $X_C$. A Casimir with nonzero relative differential can
therefore change the bracket for the same connection.

\begin{example}[An inner-curvature obstruction on the algebraic torus]
\label{ex:torus}
Let
\[
 R=k[s^{\pm1},t^{\pm1}],\qquad A=R[x,y,z,w],\qquad
 e_s=s\partial_s,\quad e_t=t\partial_t.
\]
Write $\theta_s=ds/s$ and $\theta_t=dt/t$ for the dual base forms,
and set
$P_{ab}=\partial_a\otimes\partial_b-\partial_b\otimes\partial_a$.
The commuting coordinate derivations define the Weyl--Moyal product
\cite[Section~1.3.1]{WaldmannIntro}
\[
 f\star g
   =m\exp\!\left(\frac{\hbar P_{xy}+\hbar^2P_{zw}}{2}\right)
        (f\otimes g).
\]
Its classical bracket is $\pi=\partial_x\wedge\partial_y$, and
its first divided-commutator correction is
$\beta=\partial_z\wedge\partial_w$. In particular,
\[
                  [x,y]_\star=\hbar,
 \qquad           [z,w]_\star=\hbar^2.
\]
Let $H$ differentiate base coefficients, and prescribe
$C=z\,\theta_s\wedge\theta_t$.
The connection is flat, $z$ is a classical Casimir, and
$d_HC=0$, so $(H,C)$ is a Hamiltonian unfolding.
The base independence of the product also makes $H$ a connection by
star-product derivations.

Suppose a first-order lift exists, and write
$D_{e_s}=H_{e_s}+\hbar V_s+O(\hbar^2)$ and
$D_{e_t}=H_{e_t}+\hbar V_t+O(\hbar^2)$.
The coefficient of $\hbar$ in the curvature equation, evaluated on
$w$, is
\[
 e_s(V_t(w))-e_t(V_s(w))=1.
\]
Indeed $I_z(w)=\hbar$, and the possible curving correction $c$
contributes $\{c_{st},w\}=0$. Taking the coefficient of $s^0t^0$
annihilates the left side and leaves $1$ on the right. Thus the
inner-curvature obstruction is nonzero, while the derivation
obstruction vanishes.

For the same product and connection, the pair $(H,0)$ has the full
lift $(H,0)$. The two classical pairs are also inequivalent by
splitting changes fixing $H$. Such a change has coefficients in
$Z_\pi=R[z,w]$, and its effect on the curving is $d_Hb$.
The coefficient of $s^0t^0$ in $(d_Hb)(e_s,e_t)$ is zero, whereas
that coefficient in $C(e_s,e_t)$ is $z$.
\end{example}

\subsection{Obstructions at successive orders}\label{sec:tower}

At each new order, the unknown coefficients enter the equations
linearly through the classical unfolding. This is why the same
complex $F^1$ detects successive obstructions, although the
obstruction may depend on earlier choices. For tentative series
$D=H+\sum_{i\ge1}\hbar^iD_i$ and $C_\hbar=C+\sum_{i\ge1}\hbar^ic_i$,
define the three \emph{defects}
\begin{align*}
 \mathcal{P}_\xi(a,b)&=D_\xi(a\star b)-D_\xi(a)\star b-a\star D_\xi(b),\\
 \mathcal{K}_{\xi\eta}&=[D_\xi,D_\eta]-D_{[\xi,\eta]}
                         -I_{C_{\hbar,\xi\eta}},\\
 \mathcal{J}&=d_D C_\hbar.
\end{align*}
Definition~\ref{def:partial-lift} requires the following precision
through order $r$.
\[
\begin{array}{c|c|c}
 \text{equation}&\text{required vanishing}&\text{coefficients used}\\ \hline
 \text{derivation rule}&\mathcal{P}=O(\hbar^{r+2})&D_1,\ldots,D_r; D_{r+1}\\
 \text{inner curvature}&\mathcal{K}=O(\hbar^{r+1})&D_1,\ldots,D_r; c_1,\ldots,c_r\\
 \text{Bianchi equation}&\mathcal{J}=O(\hbar^{r+1})&D_1,\ldots,D_r; c_1,\ldots,c_r
\end{array}
\]
The coefficient $D_{r+1}$ in the first row is auxiliary. For $r=0$
we may take $D_1=0$ since $B_1=\pi/2$.

Now fix a partial lift $s_{n-1}$ through order $n-1$. Choose an
admissible auxiliary $D_n$ and a two-form $c_n$ as tentative
coefficients at order $n$, setting higher coefficients to zero. The leading defects of these data are
\begin{equation}\label{eq:ob}
 A_n=[\hbar^{n+1}]\Alt \mathcal{P},\qquad
 K_n=[\hbar^n]\mathcal{K},\qquad J_n=[\hbar^n]\mathcal{J},
 \qquad \Alt \mathcal{P}_\xi(a,b)=\mathcal{P}_\xi(a,b)-\mathcal{P}_\xi(b,a).
\end{equation}
Together these terms give a degree-three cochain $O_n=(A_n,K_n,J_n)$
of $F^1$, with one, two and three horizontal arguments respectively.
At order one, for instance, $K_1=d_HV-\beta(C,-)-X_c$ includes the
curving correction, as in~\eqref{eq:four}.
The extension bracket identifies $O_n$ with the leading term of
its Jacobiator. This gives the first identity below. Changing the
tentative coefficients by $(D_n,c_n)\mapsto(D_n+W,c_n+c)$ gives
the second.
\[
                  d_LO_n=0,\qquad O'_n-O_n=d_L(W,c).
\]
Here $W_\xi$ is a relative derivation because the two choices of
$D_n$ solve the same product equation through order $n$.

\begin{definition}[Extensions of a partial lift]\label{def:partial-extension}
An \emph{extension of $s_{n-1}$ through order $n$} is a pair
$(D_n,c_n)$ for which the enlarged list of coefficients is a partial
Hamiltonian lift through order $n$. Two such extensions are equivalent
if they are related, through that order, by a splitting change
$\hbar^n\alpha$ with $\alpha\in\Omega_T^1(A)$.
\end{definition}

\Needspace{8\baselineskip}
\begin{theorem}[Extending a specified partial lift]\label{thm:tower}
The class $o_n(s_{n-1})=[O_n]\in H^3(F^1)$ is well defined and
vanishes precisely when $s_{n-1}$ extends through order $n$.
When it vanishes, the equivalence classes of extensions form a torsor
for $H^2(F^1)$. The class is natural under changes of splitting
and relative differential product equivalences.
\end{theorem}

The proof in Appendix~\ref{app:successive-obstructions} derives
the two identities and constructs the auxiliary coefficient once
$O_n$ has been cancelled.

\begin{corollary}[Vanishing of successive obstructions]\label{cor:vanishing}
The first obstruction is $o_1(H,C)=\delta[\beta]$. If $H^3(F^1)=0$,
every partial Hamiltonian lift extends to a full formal lift.
In particular, every prescribed product admits a full lift of $(H,C)$.
\end{corollary}
\begin{proof}
For $n=1$, take the tentative data $D=H$ and $C_\hbar=C$.
The coefficient of $\hbar$ in the divided commutator is $\beta$, so
\[
 O_1=\bigl(\cL_H\beta,-\beta(C,-),0\bigr)
       =d_L(\beta,0,0).
\]
The last expression uses $d_\pi\beta=0$ and is the cocycle defining
$\delta[\beta]$ in the exact sequence~\eqref{eq:restriction}.
If $H^3(F^1)=0$, Theorem~\ref{thm:tower} extends any given partial
lift by one order. Iteration gives compatible connection and curving coefficients at
every order. Each coefficient of each defect vanishes at a finite
stage, so the resulting $\hbar$-adic limits give a full lift.
\end{proof}

\begin{example}[The auxiliary coefficient on the constant Moyal plane]
\label{ex:moyal-auxiliary}
On the constant symplectic family $A=k[t][x,y]$ over $R=k[t]$,
take $\pi=\partial_x\wedge\partial_y$ and the Weyl--Moyal product
\cite{Moyal1949,WaldmannIntro}
\[
 f\star g=m\exp\!\left[
 \frac{\hbar}{2}
 (\partial_x\otimes\partial_y-\partial_y\otimes\partial_x)
 \right](f\otimes g),
 \qquad [x,y]_\star=\hbar.
\]
Here $m$ is ordinary multiplication. The product is independent
of $t$, so $H_{\partial_t}=\partial_t$ with $C=0$ is already a full lift.

Now fix the partial lift through order one with $D_1=0$.
The operator
\[
 \widehat D_{\partial_t}=\partial_t+\hbar^2 x\partial_x
\]
is an admissible auxiliary choice: $x\partial_x$ is an ordinary
derivation, so its contribution to the star-product Leibniz defect
begins at order $\hbar^3$. Curvature and Bianchi hold because
$\wedge_R^2T=0$.

To retain $D_2=x\partial_x$ at order two, we would need a further
auxiliary coefficient $U$. For any such operator put
$\widetilde D_{\partial_t}=\widehat D_{\partial_t}+\hbar^3U$.
Its commutator Leibniz defect on the coordinate functions is
\[
 \widetilde D_{\partial_t}[x,y]_\star
 -[\widetilde D_{\partial_t}x,y]_\star
 -[x,\widetilde D_{\partial_t}y]_\star
 =-\hbar^3\pmod{\hbar^4}.
\]
Indeed $\widetilde D_{\partial_t}(\hbar)=0$, the contribution of
$\hbar^2x\partial_x$ is $-\hbar^3$, and the terms containing $U$
start at order $\hbar^4$. Thus this particular choice of $D_2$
obstructs the next derivation equation. Choosing $D_2=0$ instead
extends the same order-one lift to all orders.
\end{example}

\subsection{Finite \texorpdfstring{\'etale}{etale} detection and the local--global distinction}
Proposition~\ref{prop:reductive-etale-descent} glues reductive
lifts equipped with descent arrows. Here we start with a global
partial lift on an arbitrary smooth affine Poisson family and ask
whether its obstruction can disappear on a cover. For a finite
\'etale cover, normalised trace will give a retraction on cochains.

Let $R\to R'$ be finite faithfully flat \'etale, and put
$A'=A\otimes_RR'$. Pull back $(H,C)$ and the prescribed product.
Write $L'$ for the coupling Lie algebroid of the pulled-back
Hamiltonian unfolding and $F'^1$ for its filtered subcomplex.
Base change for differentials and \'etaleness give
\[
 T'=\Der_k(R')\simeq R'\otimes_R T,
 \qquad
 \Omega^1_{A'/R'}\simeq A'\otimes_A\Omega^1_{A/R},
 \qquad L'\simeq A'\otimes_A L
\]
as modules. For the first isomorphism, dualise
$\Omega^1_{R'/k}\simeq R'\otimes_R\Omega^1_{R/k}$
\cite[Tag~0H93]{Stacks}. A base derivation $\xi$ then has a unique
extension $\xi'$ to $R'$.
Finite projectivity then identifies the graded cochains by
\[
 C^q(L';A')\simeq C^q(L;A)\otimes_RR',
 \qquad (F'^1)^q\simeq(F^1)^q\otimes_RR'.
\]

Let $e$ denote the locally constant positive rank of $R'$ over $R$.
The trace of multiplication by $1$ is $e$~\cite[Tag~0BVH]{Stacks}.
As $k$ has characteristic zero, division by $e$ gives
\[
 \operatorname{tr}_0=e^{-1}\operatorname{Tr}_{R'/R}:R'\longrightarrow R,
 \qquad \operatorname{tr}_0(1)=1.
\]

\begin{lemma}[Horizontal compatibility of the trace]
\label{lem:horizontal-trace}
For every $\xi\in T$ and $u\in R'$,
\[
                         \operatorname{tr}_0(\xi' u)=\xi(\operatorname{tr}_0(u)).
\]
\end{lemma}
\begin{proof}
Regard $R'$ as a finite projective $R$-module with connection
$\xi'$, and write $m_u$ for multiplication by $u$.
The derivation rule gives
$[\xi',m_u]=m_{\xi' u}$.
On an open set where this module is free, write
$\xi'=\xi+\Gamma$ in a chosen frame, where $\xi$ acts on the
matrix coefficients and $\Gamma$ is an $R$-linear matrix.
If $M_u$ represents $m_u$, the identity becomes
\[
       M_{\xi' u}=\xi(M_u)+[\Gamma,M_u].
\]
Taking matrix traces gives
$\operatorname{Tr}(m_{\xi'u})=\xi(\operatorname{Tr}(m_u))$,
because the trace of a commutator is zero.
Dividing by the locally constant rank $e$ gives the assertion.
\end{proof}

The scalar maps $R\to R'$ and $\operatorname{tr}_0$ now define maps on graded
cochains:
\[
 \iota(c)=c\otimes1,\qquad
 \operatorname{tr}(c\otimes u)=c\,\operatorname{tr}_0(u).
\]
They preserve the number of horizontal arguments, so restrict to
$F^1$ and $F'^1$.

\begin{proposition}[Finite \'etale detection]\label{prop:etale-detection}
Pullback induces a split injection of $k$-vector spaces
\[
                H^q(F^1)\longrightarrow H^q(F'^1)
\]
for every $q$, compatible with the obstruction classes $o_n$.
A specified partial lift extends by one order over $R$ if and only
if its pullback extends by one order over $R'$.
\end{proposition}
\begin{proof}
We verify that $\iota$ and $\operatorname{tr}$ are cochain maps.
The Chevalley--Eilenberg differential is determined by the anchor
actions and the brackets on the generators
$\dd f$ and $\widehat\xi$. Relative anchors kill the new base
coefficients, so their actions commute with $\operatorname{tr}$.
For a horizontal generator, its action on $a\otimes u$ is
\[
 H'_\xi(a\otimes u)=H_\xi(a)\otimes u+a\otimes\xi'(u).
\]
Lemma~\ref{lem:horizontal-trace} therefore gives
\[
 (\id_A\otimes\operatorname{tr}_0)H'_\xi(a\otimes u)
       =H_\xi\bigl(a\operatorname{tr}_0(u)\bigr).
\]
The brackets of the pulled-back generators are the pullbacks of
their original brackets, so the bracket terms in the differential
also commute with trace. The same identities with $u=1$ prove
compatibility for $\iota$. Thus
\[
\begin{CD}
(F^1)^q @>{\iota}>> (F'^1)^q @>{\operatorname{tr}}>> (F^1)^q\\
@V{d_L}VV @V{d_{L'}}VV @V{d_L}VV\\
(F^1)^{q+1} @>{\iota}>> (F'^1)^{q+1}
                         @>{\operatorname{tr}}>> (F^1)^{q+1}
\end{CD}
\]
commutes, and $\operatorname{tr}\circ\iota=\id$ because $\operatorname{tr}_0(1)=1$.
Passing to cohomology gives the split injection.

Each defect in~\eqref{eq:ob} is formed from the prescribed product,
connection, curving, and their brackets. Pulling back tentative lift
data therefore pulls back their defect cochain. Hence
$o_n(s_{n-1}')=\iota_*o_n(s_{n-1})$ for the pulled-back partial
lift $s_{n-1}'$. Injectivity, together with
Theorem~\ref{thm:tower}, proves the extension criterion.
\end{proof}

On a Zariski open cover, a global obstruction can vanish on every
member. The following surface realises it as a de~Rham class.

\begin{example}[Local full lifts without a global first-order lift]
\label{ex:local-global-curvature}
Consider the smooth affine quadric
\[
 R=\mathbb C[x,y,z]/(xy-z^2+1),\qquad S=\Spec R.
\]
The gradient $(y,x,-2z)$ has no common zero on $S$, proving
smoothness. The principal opens
$U_+=D(z+1)$ and $U_-=D(z-1)$ cover $S$, and carry the regular
one-forms
\[
 \alpha_+=-\frac{y\,dx}{z+1},\qquad
 \alpha_-=-\frac{y\,dx}{z-1}.
\]
On $D(x)$ the equation of the surface gives
$y=(z^2-1)/x$, and direct differentiation yields
\[
 d\alpha_+=d\alpha_-=\frac{dx\wedge dz}{x}.
\]
Since $xy$ is invertible on $U_+\cap U_-$, the formula holds on
the whole overlap. The forms $d\alpha_\pm$ therefore glue to a
regular two-form $\omega$ on $S$.

We first verify that $[\omega]\ne0$ in algebraic de~Rham cohomology.
Give $x,y,z$ weights $1,-1,0$, respectively, and give their
differentials the same weights. The differential preserves weights,
and $\omega$ has weight zero. A global primitive would therefore
have a weight-zero component that is still a primitive. The
weight-zero, weight-one and weight-minus-one pieces of $R$ are
\[
 R_0=\mathbb C[z],\qquad
 R_1=x\mathbb C[z],\qquad R_{-1}=y\mathbb C[z].
\]
It follows that a weight-zero one-form is a sum of polynomial
multiples of $y\,dx$, $x\,dy$ and $dz$.
The relation $y\,dx+x\,dy=2z\,dz$ puts every such form in the form
\[
                 \alpha=P(z)y\,dx+Q(z)\,dz,
                 \qquad P,Q\in\mathbb C[z].
\]
On $D(x)$ one then computes
\[
               d\alpha=-\bigl((z^2-1)P(z)\bigr)'\omega.
\]
The equality $d\alpha=\omega$ would give
$(z^2-1)P(z)=-z+c$ for some constant $c$.
Evaluating at $z=1$ and $z=-1$ forces $c=1$ and $c=-1$,
respectively. Hence $\omega$ has no global regular primitive.

Now take the relative affine plane $A=R[u,v]$ with zero classical
Poisson bracket, and prescribe the Weyl--Moyal product
\cite[Section~1.3.1]{WaldmannIntro}
\[
 f\star g=m\exp\!\left[
 \frac{\hbar^2}{2}
 (\partial_u\otimes\partial_v-\partial_v\otimes\partial_u)
 \right](f\otimes g).
\]
For the coordinate $u$ the commutator is given by
$[u,f]_\star=\hbar^2\partial_v f$, so $I_u=\hbar\partial_v$.
Let $H$ differentiate the base coefficients and put $C=u\omega$.
The classical identities hold because $H$ is flat, the classical
Poisson bracket is zero, and $d_HC=u\,d\omega=0$.

On each of the two opens, define
\[
       D_\pm=H+\hbar\alpha_\pm\partial_v,
       \qquad C_\hbar=u\omega.
\]
The corrections are base multiples of the product derivation
$\partial_v$ and commute with one another. Since $H$ differentiates
base coefficients,
\[
 F_{D_\pm}=\hbar(d\alpha_\pm)\partial_v
          =\hbar\omega\partial_v=I_{u\omega},
 \qquad d_{D_\pm}(u\omega)=0.
\]
Thus these are full formal lifts on $U_+$ and $U_-$.

A global first-order lift would have
$D=H+\hbar V+O(\hbar^2)$, with a possible curving correction
$C_\hbar=u\omega+\hbar c+O(\hbar^2)$.
Since the divided commutator begins at order $\hbar$, the term
$\hbar c$ contributes to the inner derivation only from order
$\hbar^2$ onwards. Applying the order-$\hbar$ curvature equation
to $v$ therefore gives
\[
                         d_H\bigl(V(v)\bigr)=\omega.
\]
Here $V(v)\in\Omega_T^1(A)$. The evaluation
$\varepsilon:A\to R$, $u,v\mapsto0$, commutes with the base
connection, hence gives a morphism of complexes
\[
\begin{CD}
\Omega_T^1(A) @>{d_H}>> \Omega_T^2(A)\\
@V{\varepsilon}VV @V{\varepsilon}VV\\
\Omega^1_{R/\mathbb C} @>{d}>> \Omega^2_{R/\mathbb C}.
\end{CD}
\]
It would send $V(v)$ to a global primitive of $\omega$, contradicting
the calculation above. Thus the global first-order obstruction is
nonzero despite the full lifts over $R_{z+1}$ and $R_{z-1}$.
\end{example}

\section{The Heisenberg family}\label{sec:heis}

We apply the obstruction theory of Section~\ref{sec:quantum} to a
family whose fibres are duals of the Heisenberg Lie algebra. The
classical connection $H$ is flat and already preserves the prescribed
star product. The question is whether its chosen curvature Hamiltonians
can also be lifted so that they remain covariantly closed. We take
them to be $f(z)\omega$, where $z$ is the linear Casimir coordinate,
$f$ is a polynomial, and $\omega$ is a fixed closed two-form on the
parameter torus.

\medskip
\noindent
We first allow quantum corrections to $H$. On the whole family,
the Bianchi equation obstructs a first-order lift whenever $f$ is
nonconstant. On the regular locus $z\ne 0$, we construct a full lift
for every $f$. Subsection~5.1 computes the cohomology of the quantum
centre and explains why, if we keep $H$ fixed, the obstruction for
nonconstant $f$ persists even on the regular locus.
Subsection~\ref{sec:kontsevich-comparison} identifies the same
obstruction for the Kontsevich product of a formal deformation of
the Heisenberg bracket and interprets it in terms of polynomial
boundary observables of the Poisson sigma model~\cite{CF}.

Work over a field $k$ of characteristic zero and put
\[
 R=k[s^{\pm1},t^{\pm1},u^{\pm1}],\qquad
 S=\Spec R=(\mathbb G_{m,k})^3.
\]
Let $\mathfrak h_3$ be the Heisenberg Lie algebra with basis
$v_x,v_y,v_z$, bracket $[v_x,v_y]=v_z$ and central element $v_z$.

\begin{samepage}
\begin{definition}[Heisenberg Poisson family]\label{def:heisenberg-family}
The \emph{Heisenberg Poisson family over $S$} is the projection
\[
 p:\mathcal X=S\times\mathfrak h_3^\vee\longrightarrow S
\]
equipped with the fibrewise linear Poisson structure induced by
$\mathfrak h_3$. Its coordinate algebra and relative bivector are
\[
 A=\Sym_R(R\otimes_k\mathfrak h_3)=R[x,y,z],
 \qquad \pi=z\partial_x\wedge\partial_y.
\]
Here $x,y,z$ are the linear functions corresponding to $v_x,v_y,v_z$.
In particular $\{x,y\}=z$ and $z$ is a Casimir.
\end{definition}
\end{samepage}

The Poisson rank in each fibre is two on $z\ne0$ and zero on $z=0$.

On the base, the tangent module $T=\Der_kR$ has the Euler basis
\[
 e_s=s\partial_s,\qquad e_t=t\partial_t,\qquad e_u=u\partial_u.
\]
Its dual basis consists of the algebraic logarithmic forms
$\theta_r=dr/r$ for $r\in\{s,t,u\}$. We will use
\[
 \omega=\theta_t\wedge\theta_u,\qquad
 \nu=\theta_s\wedge\theta_t\wedge\theta_u.
\]
In particular, every Bianchi defect is a multiple of $\nu$.

To prescribe the quantum product, take the commuting relative
derivations on $A[[\hbar]]$ given by
\[
 P_\hbar=z\partial_x+\hbar z\partial_z,
 \qquad Y=\partial_y,
 \qquad [P_\hbar,Y]=0.
\]
With $m$ denoting ordinary multiplication, define
\begin{equation}\label{eq:heisproduct}
 g\star h
 =m\exp\!\left[\frac{\hbar}{2}
       (P_\hbar\otimes Y-Y\otimes P_\hbar)\right](g\otimes h).
\end{equation}
Commutation gives associativity~\cite[Theorem~2.1]{GZ}. The
exponential has relative bidifferential coefficients independent of
$s,t,u$, and $1$ is its unit.

The divided commutator has the expansion
\[
 \mu(g,h)=\{g,h\}+\hbar\beta(g,h)+O(\hbar^2),
 \qquad \beta=z\partial_z\wedge\partial_y.
\]
Indeed, the linear exponential term contributes $P_\hbar\wedge
Y=\pi+\hbar\beta$, and the higher odd powers contribute to $\mu$
from order $\hbar^2$. With $y$ as one argument only the linear term
survives, so
\[
 [g,y]_\star=\hbar P_\hbar g,
 \qquad
 [x,y]_\star=\hbar z,\quad
 [z,y]_\star=\hbar^2z,\quad
 [x,z]_\star=0.
\]

For transport, take the $R$-linear connection $H$ given by
\[
 H_s:=H_{e_s}=e_s+\partial_x,\qquad
 H_t:=H_{e_t}=e_t,\qquad H_u:=H_{e_u}=e_u.
\]
Thus transport in the $s$-direction includes translation in the fibre
coordinate $x$. The three displayed fields commute, and each commutes
with $P_\hbar$ and $Y$. Consequently $H$ is flat and already acts by
derivations of the prescribed star product.

Choose $f\in k[z]$ and prescribe the curving $C_f=f(z)\omega$.
Since $z$ is a Casimir fixed by $H$ and $\omega$ is closed,
$F_H=0=X_{C_f}$ and $d_HC_f=0$. Thus $(H,C_f)$ is a Hamiltonian
unfolding.

First take $H$ itself as the quantum connection. Its curvature
vanishes, so we seek a central lift of $f(z)$. The coordinate $z$ satisfies
$P_\hbar(z)=\hbar z$, while $P_\hbar(x)=z$. Subtracting $\hbar x$
therefore gives an invariant of both derivations in the product,
\[
 q=z-\hbar x,\qquad P_\hbar(q)=Y(q)=0.
\]
This suggests replacing the classical curving $f(z)\omega$ by
$f(q)\omega$.

\begin{lemma}[A central curving and its Bianchi defect]
\label{lem:heis-central-curving}
The form $f(q)\omega$ is a quantum-central lift of $C_f$, and
\begin{equation}\label{eq:drift}
 F_H=I_{f(q)\omega}=0,\qquad
 d_H(f(q)\omega)=-\hbar f'(q)\nu.
\end{equation}
Consequently the derivation and inner-curvature tests for a first-order
lift of $(H,C_f)$ are satisfied, with corrections
$V=0$ and $c=-xf'(z)\omega$.
\end{lemma}
\begin{proof}
Since $P_\hbar$ and $Y$ annihilate $f(q)$, all positive-order
terms in multiplication by $f(q)$ vanish. It is therefore central
and generates the zero curvature of $H$.
The connection acts on the new coordinate by
\[
                  H_s(q)=-\hbar,\qquad H_t(q)=H_u(q)=0.
\]
Closedness of $\omega$ now gives the Bianchi formula.
The expansion $f(q)=f(z)-\hbar xf'(z)+O(\hbar^2)$ gives the
stated corrections and the leading Bianchi coefficient $-f'(z)\nu$.
\end{proof}

Could another correction cancel $-f'(z)\nu$? By
Proposition~\ref{prop:first-order-tests}, the possible changes of
its class form $\im\tau\subset H^3(T;Z_\pi)$. They are represented
by $d_Hb+W\cdot C_f$, where $W$ is a one-form of relative Poisson
vector fields and $d_HW=X_b$. Laurent coefficients on the base let
us compute this image.

Write $\mathrm{ct}:A\to k[x,y,z]$ for extraction of the coefficient
of $s^0t^0u^0$, and $\mathrm{ct}_{t,u}:A\to k[s^{\pm1},x,y,z]$
for extraction of the coefficient of $t^0u^0$.
On top forms with central coefficients, define
\[
 \operatorname{res}:\Omega_T^3(Z_\pi)\longrightarrow k[z],
 \qquad \operatorname{res}(g\nu)=\mathrm{ct}(g).
\]
Here $Z_\pi=R[z]$: the Casimir equations are
$z\partial_xg=z\partial_yg=0$, and the domain property of $A$
reduces them to independence of $x$ and $y$.
Since $z$ is a Casimir fixed by $H$, the coupling differential and
hence the transgression are $k[z]$-linear.

\begin{lemma}[Laurent residue and the transgression image]
\label{lem:heis-residue-transgression}
The residue induces an isomorphism
\[
 \operatorname{res}_*:H^3(T;Z_\pi)\xrightarrow{\ \sim\ }k[z].
\]
Under this isomorphism, the image of $\tau$ is the ideal
$zf'(z)k[z]$.
\end{lemma}
\begin{proof}
On $Z_\pi=R[z]$, the three Euler derivatives kill the constant
Laurent coefficient, so every exact top form has zero residue.
Conversely, take a Laurent monomial
$s^at^bu^c g(z)\nu$ with $(a,b,c)\ne(0,0,0)$.
Choose $r\in\{s,t,u\}$ whose exponent $n_r$ is nonzero.
Contraction with $e_r$, denoted by $\iota_{e_r}$, gives the primitive
\[
 d_H\left(\frac{s^at^bu^c g(z)}{n_r}\,
                         \iota_{e_r}\nu\right)
                 =s^at^bu^c g(z)\nu.
\]
Each Laurent polynomial is a finite sum of monomials. Summing their
primitives proves that $\ker\operatorname{res}=\im d_H$, and
$g(z)\nu$ has residue $g(z)$. It remains to calculate the residues
of the transgression.

Let $(W,b)$ represent an element in the transgression formula, and
write $W_r=W(e_r)$ and $b_{rr'}=b(e_r,e_{r'})$.
The $(t,u)$ component of $d_HW=X_b$ is
\[
               [e_t,W_u]-[e_u,W_t]=X_{b_{tu}}.
\]
The operators on the left differentiate Laurent coefficients in
$t$ and $u$. Taking their constant coefficient therefore gives
\[
 X_{\mathrm{ct}_{t,u}(b_{tu})}=0,
 \qquad
 \mathrm{ct}_{t,u}(b_{tu})\in k[s^{\pm1},z].
\]
In particular, this partial constant coefficient is independent of
$x$. On the other hand,
\[
 (d_Hb)_{stu}
 =(e_s+\partial_x)b_{tu}-e_t b_{su}+e_u b_{st}.
\]
Only the translation term can contribute to the full constant
coefficient, and the preceding Casimir calculation gives
\[
 \mathrm{ct}\bigl((d_Hb)_{stu}\bigr)
              =\mathrm{ct}(\partial_xb_{tu})=0.
\]
Thus the residue of the transgression is
\[
 \operatorname{res}(d_Hb+W\cdot C_f)
                    =f'(z)\,\mathrm{ct}\bigl(W_s(z)\bigr).
\]
A Poisson vector field preserves Casimirs, so $W_s(z)\in R[z]$.
Applying its derivation rule to $\{x,y\}=z$ also gives
\[
 W_s(z)=z\bigl(\partial_xW_s(x)+\partial_yW_s(y)\bigr).
\]
Hence $W_s(z)\in zR[z]$, and the residue belongs to $zf'(z)k[z]$.

Every element of this ideal is realised. For $g\in k[z]$, take
\[
 W_s=g(z)(z\partial_z+x\partial_x),\qquad W_t=W_u=0,
 \qquad b=0.
\]
The vector field $W_s$ preserves $\pi$, because its action on the
coefficient $z$ cancels its action on $\partial_x$. Its coefficients
are independent of $t,u$, so $d_HW=0$. The transgression of this
pair is $zf'(z)g(z)\nu$. This proves the asserted image.
\end{proof}

\begin{proposition}[The last first-order obstruction]\label{prop:third}
The class $[\beta]$ survives to the third page of the base-form
spectral sequence for $(H,C_f)$, where
\begin{equation}\label{eq:fobstruction}
 E_3^{3,0}\simeq k[z]\big/(zf'(z)),\qquad
 d_3[\beta]=[-f'(z)].
\end{equation}
The unfolding $(H,C_f)$ admits a first-order Hamiltonian lift on
the whole family precisely when $f$ is constant.
\end{proposition}
\begin{proof}
Lemma~\ref{lem:heis-central-curving} solves the first two equations,
so $d_1[\beta]=d_2[\beta]=0$. Its remaining Bianchi coefficient is
$-f'(z)\nu$. By Lemma~\ref{lem:heis-residue-transgression}, the
quotient by all allowable changes of this class is precisely
$k[z]/(zf'(z))$. Proposition~\ref{prop:first-order-tests} identifies
the resulting residue class with $d_3[\beta]$.

For nonconstant $f$ we have $f'\ne0$. If $f'=zf'g$ for some
$g\in k[z]$, cancellation would give $1=zg$. Hence the obstruction
is nonzero. For constant $f$, $(H,C_f)$ itself satisfies both
quantum equations to all orders.
\end{proof}

\begin{remark}[Support of the obstruction]
\label{rem:heis-obstruction-support}
For nonconstant $f$, multiplication by a polynomial $r(z)$ kills
$[-f'(z)]$ exactly when $rf'\in zf'k[z]$. Cancellation of $f'$
therefore gives
\[
 \operatorname{Ann}_{k[z]}([-f'])=(z),\qquad
 k[z]\cdot[-f']\simeq k[z]/(z).
\]
The submodule generated by the obstruction is supported at $z=0$.
This conclusion also holds when $f'(0)=0$.
\end{remark}

We can see the effect of passing to the regular locus directly in
the transport equations. Extend the product formula to
$A[z^{-1}][[\hbar]]$. The coordinate $q=z-\hbar x$ has inverse
\[
 q^{-1}=z^{-1}\sum_{j\ge0}(\hbar x/z)^j,
\]
which is central and multiplies ordinarily. Replace the translation
part of $H_s$ by an inner derivation, defining
\begin{equation}\label{eq:regularlift}
 D_s=e_s+I_{-y/q}=e_s+q^{-1}P_\hbar,
 \qquad D_t=e_t,\qquad D_u=e_u.
\end{equation}

\begin{corollary}[Full lifts on the regular locus]
\label{cor:heis-regular-lift}
For every $f\in k[z]$, the pair $(D,f(q)\omega)$ is a full
Hamiltonian lift of $(H,C_f)$ over $A[z^{-1}]$.
\end{corollary}
\begin{proof}
Since $[g,y]_\star=\hbar P_\hbar g$ and $q^{-1}$ multiplies
ordinarily, $I_{-y/q}=q^{-1}P_\hbar$. Thus $D_s$ is a derivation
reducing to $e_s+\partial_x$. Its coefficients are independent of
$t,u$, so $D_s,D_t,D_u$ commute and $F_D=0$. All three annihilate
$q$. The central two-form $f(q)\omega$ therefore generates $F_D$,
is $D$-closed and reduces to $C_f$.
\end{proof}

The same mechanism can be placed at any prescribed formal order.
For $m\ge1$, set
\[
 P_\hbar^{(m)}=z\partial_x+\hbar^m z\partial_z,
 \qquad q_m=z-\hbar^m x,
\]
and define $\star_m$ by the exponential formula
\eqref{eq:heisproduct} with $P_\hbar^{(m)}$ in place of $P_\hbar$.
The connection $H$ preserves every $\star_m$, and $q_m$ is central
for that product. Fix the partial lift $s_{m-1}$ whose connection
and curving corrections through order $m-1$ are all zero.

\begin{corollary}[Delayed Bianchi obstruction]\label{cor:heis-delayed}
For the prescribed product $\star_m$, the obstruction to extending
$s_{m-1}$ through order $m$ is
\[
           o_m(s_{m-1})=[(0,0,-f'(z)\nu)]\in H^3(F^1).
\]
This class is nonzero for every nonconstant $f$.
\end{corollary}
\begin{proof}
Since $P_\hbar^{(m)}(q_m)=Y(q_m)=0$, the tentative pair
$(H,f(q_m)\omega)$ satisfies multiplicativity and inner curvature
exactly. Its expansion and Bianchi defect are
\[
 \begin{aligned}
 f(q_m)&=f(z)-\hbar^mxf'(z)+O(\hbar^{2m}),\\
 d_H(f(q_m)\omega)&=-\hbar^m f'(q_m)\nu.
 \end{aligned}
\]
Thus the coefficients through order $m-1$ give the specified
partial lift, and the leading defect cochain is $(0,0,-f'(z)\nu)$.
Theorem~\ref{thm:tower} gives its class.

If this cochain were a boundary $d_L(W,c)$, its first two
components in~\eqref{eq:FilDiff} would give
\[
                  \cL_W\pi=0,\qquad d_HW=X_c.
\]
Lemma~\ref{lem:heis-residue-transgression} would then put the
residue of its horizontal component in $zf'(z)k[z]$.
For nonconstant $f$, the polynomial $-f'(z)$ lies outside that
ideal by the cancellation argument in Proposition~\ref{prop:third}.
The obstruction class is therefore nonzero.
\end{proof}

\subsection{Cohomology of the quantum centre}

We now fix the quantum connection $D=H$. The preliminary curving
$f(q)\omega$ generates its curvature, so Theorem~\ref{thm:central-bianchi}
places the remaining obstruction in $H^3(T;\hbar Z_\hbar)$. To
evaluate it we first determine the quantum centre and the action of
$H$ on it.

For $q=z-\hbar x$, write $R[q][[\hbar]]$ for the image of the
substitution
\[
 R[\zeta][[\hbar]]\longrightarrow A_\hbar,
 \qquad \sum_{i\ge0}\hbar^i g_i(\zeta)
            \longmapsto\sum_{i\ge0}\hbar^i g_i(q).
\]
This substitution is well defined coefficientwise, and its reduction
modulo $\hbar$ identifies $\zeta$ with $z$.

\begin{lemma}[The quantum centre and its connection]\label{lem:heis-centre}
The substitution map identifies the quantum centre with
$Z_\hbar=R[q][[\hbar]]$. In these coordinates, the restricted
connection is
\[
 H_s=e_s-\hbar\partial_q,\qquad H_t=e_t,\qquad H_u=e_u.
\]
\end{lemma}
\begin{proof}
Lemma~\ref{lem:heis-central-curving} proves that $R[q][[\hbar]]$
is central. Conversely, a central series reduces to a Casimir in $R[z]$.
If $a\bmod\hbar=g_0(z)$, then $(a-g_0(q))/\hbar$ is central by
$\hbar$-torsion-freeness. Iteration gives
$a=\sum_{i\ge0}\hbar^i g_i(q)$.
The first nonzero coefficient in a difference of two expansions
survives reduction at that order, proving uniqueness. The connection
formulas follow from $H_s(q)=-\hbar$, $H_t(q)=H_u(q)=0$ and the
Euler action on base coefficients.
\end{proof}

Set
\[
 M=\hbar Z_\hbar,\qquad
 V_\theta=k\theta_s\oplus k\theta_t\oplus k\theta_u,
 \qquad W=k\theta_t\oplus k\theta_u.
\]
The closed Euler coframe identifies the central cochain complex as
\[
 \Omega_T^\bullet(M)=M\otimes_k\bigwedge^\bullet V_\theta,
 \qquad
 d_H=\theta_s\wedge(e_s-\hbar\partial_q)
        +\theta_t\wedge e_t+\theta_u\wedge e_u,
\]
where each differential operator acts on the coefficient of a form.
Let $N=\hbar k[q][[\hbar]]$ and let
$C_{\mathrm{ct}}^\bullet=N\otimes_k\bigwedge^\bullet V_\theta$ be the subcomplex of
forms with constant Laurent coefficients in $s,t,u$.
Its differential is $d_{\mathrm{ct}}=-\hbar\theta_s\wedge\partial_q$.
Coefficientwise extraction of the constant Laurent term gives a
projection $\rho:\Omega_T^\bullet(M)\to C_{\mathrm{ct}}^\bullet$.

\begin{lemma}[Contraction of the nonconstant Laurent terms]
\label{lem:heis-laurent-contraction}
The inclusion $\iota:C_{\mathrm{ct}}^\bullet\to\Omega_T^\bullet(M)$ and the
projection $\rho$ are mutually inverse up to an
$\hbar$-adically continuous cochain homotopy.
\end{lemma}
\begin{proof}
In the Laurent sector $s^at^bu^c$, the differential is
\[
 d_{a,b,c}(\alpha)=
   \theta_s\wedge(a-\hbar\partial_q)\alpha
     +b\theta_t\wedge\alpha+c\theta_u\wedge\alpha.
\]
Write $\iota_r$ for contraction by the Euler field $e_r$ on the
exterior factor. The identities between exterior multiplication and
contraction give
$d_{a,b,c}\iota_t+\iota_t d_{a,b,c}=b\,\id$ and the analogous
identity with $u$ and $c$. Thus $b^{-1}\iota_t$ contracts the sector
when $b\ne0$; when $b=0$ and $c\ne0$, use $c^{-1}\iota_u$.
For $b=c=0$ and $a\ne0$, the coefficient operator is invertible,
with
\[
 (a-\hbar\partial_q)^{-1}
       =a^{-1}\sum_{r\ge0}(\hbar\partial_q/a)^r.
\]
The operator $(a-\hbar\partial_q)^{-1}\iota_s$ then contracts
that sector.

Take these contractions in the nonconstant sectors and zero in the
constant sector to define $K$. At each formal order the inverse
series uses only finitely many earlier coefficients, each a finite
Laurent polynomial. Thus $K$ acts on the completed module and
preserves its $\hbar$-adic filtration. The sector identities give
\[
 d_HK+Kd_H=\id-\iota\rho,
 \qquad \rho\iota=\id.
\]
This proves the asserted homotopy equivalence.
\end{proof}

In the constant sector, group forms according to whether they contain
$\theta_s$. The resulting families of representatives in degree $j$ are
\[
 a(\hbar)\eta,
 \qquad \hbar p(q)\theta_s\wedge\zeta,
\]
where $a(\hbar)\in\hbar k[[\hbar]]$,
$\eta\in\bigwedge^jW$, $p(q)\in k[q]$, and
$\zeta\in\bigwedge^{j-1}W$. Put $\bigwedge^jW=0$ for
$j\notin\{0,1,2\}$ and let $\hbar$ act by zero on the copy of $k[q]$
in the following formula.

\begin{proposition}[Quantum-centre cohomology]
\label{prop:heis-centre-cohomology}
The displayed representatives induce, for every $j\ge0$, an
isomorphism of $k[[\hbar]]$-modules
\[
 H^j(T;M)\simeq
 \left(\hbar k[[\hbar]]\otimes_k\bigwedge^jW\right)
 \oplus
 \left(k[q]\otimes_k\bigwedge^{j-1}W\right).
\]
\end{proposition}
\begin{proof}
By Lemma~\ref{lem:heis-laurent-contraction}, it suffices to compute
$C_{\mathrm{ct}}^\bullet$. The decomposition $V_\theta=k\theta_s\oplus W$ writes this
complex as the tensor product of the two-term complex
\[
 N\xrightarrow{-\hbar\partial_q}N\theta_s
\]
in degrees zero and one with the exterior algebra
$\bigwedge^\bullet W$, equipped with zero differential.
Coefficientwise differentiation in $q$ is surjective on
$k[q][[\hbar]]$: a primitive is obtained by integrating each
polynomial coefficient, using characteristic zero. Consequently the
kernel and image of the displayed differential are
\[
 \ker(-\hbar\partial_q)=\hbar k[[\hbar]],
 \qquad
 \im(-\hbar\partial_q)=\hbar^2 k[q][[\hbar]]\theta_s.
\]
The cokernel is $k[q]$, via $[\hbar p_0(q)+O(\hbar^2)]\mapsto
p_0(q)$, and is annihilated by $\hbar$. Tensoring the kernel and
cokernel with $\bigwedge^\bullet W$ gives the two summands and the
displayed representatives.
\end{proof}

Since $\dim W=2$, only the second summand contributes in degree
three. The isomorphism is
\[
 \rho_3:H^3(T;M)\longrightarrow k[q],\qquad
 [a\nu]\longmapsto[\hbar]\rho(a).
\]
It takes the constant Laurent coefficient and then its coefficient
of $\hbar$. By~\eqref{eq:drift} it sends $\mathfrak b(H,C_f)$ to
$-f'(q)$. Thus
\[
 \mathcal B:k[z]/k\longrightarrow H^3(T;\hbar Z_\hbar),
 \qquad [f]\longmapsto\mathfrak b(H,C_f)
\]
is well defined.

\begin{corollary}[Realisation of central Bianchi classes]
\label{cor:heis-anomaly-universal}
The map $\mathcal B$ is an isomorphism of $k$-vector spaces.
The coupling brackets defined by $C_f$ and $C_g$ agree on the
given split module precisely when $f-g$ is constant.
\end{corollary}
\begin{proof}
The preceding computation gives the commutative diagram
\[
 \begin{CD}
 k[z]/k @>{\mathcal B}>> H^3(T;\hbar Z_\hbar)\\
 @V{[f]\mapsto-f'(q)}VV @VV{\rho_3}V\\
 k[q] @= k[q].
 \end{CD}
\]
The left map is bijective by polynomial integration in characteristic
zero, and $\rho_3$ is the isomorphism just computed.

For the coupling assertion, the only potentially nonzero Euler
component of the relative differential is
\[
 (d_{\mathrm v}C_f)(e_t,e_u)=f'(z)\,d_{\mathrm v}z.
\]
By Proposition~\ref{prop:curving-information} the brackets agree
exactly when $(f'-g')(z)d_{\mathrm v}z=0$. As $d_{\mathrm v}z$ is a
basis element, this says $f-g$ is constant.
\end{proof}

On the regular locus $q$ is invertible, and the reduction argument
of Lemma~\ref{lem:heis-centre} gives $Z_\hbar=R[q,q^{-1}][[\hbar]]$
there. For the fixed connection $H$, the constant Laurent part of
the differential on $\hbar Z_\hbar$ is
$-\hbar\theta_s\wedge\partial_q$. Its image starts in order $\hbar^2$,
so $-\hbar f'(q)\nu$ still gives a nonzero class for nonconstant $f$.
With the connection~\eqref{eq:regularlift}, all three horizontal
fields annihilate $q$, so the same curving $f(q)\omega$ is closed.
The table identifies the data held fixed in each calculation.

\begin{center}
\renewcommand{\arraystretch}{1.15}
\begin{tabular}{@{}p{.46\textwidth}p{.47\textwidth}@{}}
\toprule
Fixed data & Obstruction \\
\midrule
Classical unfolding and prescribed product &
$\delta[\beta]\in H^3(F^1)$, tested by $d_1,d_2,d_3$ \\
Partial lift $s_{n-1}$ through order $n-1$ & $o_n(s_{n-1})\in H^3(F^1)$ \\
Quantum connection $D$ and classical $C$, with a preliminary
curving lifting $C$ and generating $F_D$ &
$\mathfrak b(D,C)\in H^3(T;\hbar Z_\hbar)$ \\
\bottomrule\end{tabular}
\end{center}

\subsection{Comparison with the Kontsevich product}
\label{sec:kontsevich-comparison}

The perturbative boundary product of the Poisson sigma model is
the Kontsevich product~\cite{CF}. To transfer the calculation we
need an equivalence with~\eqref{eq:heisproduct} that also respects
$H$ and the curving. For linear Poisson tensors, Dito compares the
Kontsevich and Gutt products~\cite[Corollary~1]{Dito}. In our family
we first separate the coefficient of $\beta$ from the quantisation
parameter by writing $\pi_a=\pi+a\beta$. We specialise to $a=\hbar$
after constructing the comparison.

Work over $k=\mathbb R$ for the real Kontsevich graph weights.
The coefficient rings are $R=\mathbb R[s^{\pm1},t^{\pm1},u^{\pm1}]$,
$A=R[x,y,z]$ and $A_a=A[a]$. The connection $H$ differentiates in
the original base directions and fixes $a$. All constructions extend
to $\mathbb C$ by scalar extension.

\begin{definition}[Auxiliary linear Poisson family]
\label{def:heis-auxiliary-family}
Let $\mathfrak g_a$ be the Lie algebra over $\mathbb R[a]$ whose
underlying module is free with basis $X,Y_0,Z$ and whose brackets are
\[
 [X,Y_0]_a=Z,\qquad [Z,Y_0]_a=aZ,\qquad [X,Z]_a=0.
\]
Its linear Poisson tensor, extended to $A_a$, is
\[
 \pi_a=z(\partial_x+a\partial_z)\wedge\partial_y
        =\pi+a\beta.
\]
\end{definition}

The span of $X,Z$ is an abelian ideal on which $Y_0$ acts
linearly, giving Jacobi. Put $P_a=z\partial_x+az\partial_z$.
Since $P_a$ commutes with $Y=\partial_y$ and $\pi_a=P_a\wedge Y$,
formula~\eqref{eq:heisproduct} with $P_a$ in place of $P_\hbar$
defines the exponential product $\star_{\mathrm{exp},a}$.

For the second product, use the affine Kontsevich formula
\cite[Sect.~2]{KonFormality}, normalised by
\[
 g\star_{\mathrm K,a}h
 =gh+\sum_{n\ge1}\hbar^n B^{\mathrm K}_{n,a}(g,h),
 \qquad B^{\mathrm K}_{1,a}=\tfrac12\pi_a.
\]
At order $n$, $B^{\mathrm K}_{n,a}$ is a finite weighted sum of
graph operators with $n$ copies of $\pi_a$. Each edge differentiates
a tensor coefficient or an argument in one of $x,y,z$.
Consequently these are $R[a]$-relative bidifferential operators with
polynomial coefficients, of degree at most $n$ in $a$.

Specialising $a=\hbar$ recovers~\eqref{eq:heisproduct} on the
exponential side. Write $\star_{\mathrm K}$ for the specialised graph
product of $\pi_\hbar=\pi+\hbar\beta$. At order $N$ only the terms
$\hbar^na^{N-n}$ with $0\le n\le N$ contribute.

Both products have the same commutators on the linear coordinate
functions. For the exponential product this follows directly from
$Y^2(x)=Y^2(y)=Y^2(z)=0$ and $YP_a(x)=YP_a(y)=YP_a(z)=0$.
For the Kontsevich product it is the linear-commutator identity
\cite[Lemma~3]{Dito}, with the formal parameter rescaled so that the
first bidifferential coefficient is $\pi_a/2$. The identity extends
to $\mathbb R[a]$ because each coefficient is polynomial in $a$
and it holds at every real value of $a$. Thus
\begin{equation}\label{eq:comparison-linear-relations}
 [x,y]=\hbar z,\qquad [z,y]=\hbar az,\qquad [x,z]=0
\end{equation}
for both products. 
Let $U_{\hbar,a}$ be the $\hbar$-adic completion of the enveloping
algebra over $R[a,\hbar]$ with relations
\eqref{eq:comparison-linear-relations}. The PBW symmetrisation map
identifies its underlying complete $R[a][[\hbar]]$-module with
$A_a[[\hbar]]$. Sending the generators to the linear coordinate functions gives
the arrows $J_{\mathrm{exp},a}$ and $J_{\mathrm K,a}$ below.
Both reduce to the identity under PBW, so successive inversion in
$\hbar$ makes them isomorphisms. Define the \emph{generator comparison}
$\Phi_a=J_{\mathrm K,a}J_{\mathrm{exp},a}^{-1}$.
\[
\begin{CD}
 U_{\hbar,a} @>{J_{\mathrm{exp},a}}>>
                 (A_a[[\hbar]],\star_{\mathrm{exp},a})\\
 @V{J_{\mathrm K,a}}VV @VV{\Phi_a}V\\
 (A_a[[\hbar]],\star_{\mathrm K,a})
          @= (A_a[[\hbar]],\star_{\mathrm K,a})
\end{CD}
\]
In particular, $\Phi_a$ fixes $1,x,y,z$. To track the central curving,
put $q_a=z-ax$.

\begin{lemma}[Differential comparison of the two products]
\label{lem:kontsevich-differential-comparison}
The generator comparison $\Phi_a$ commutes with $H$ and is a relative
differential equivalence of the form
\[
 \Phi_a=\id+\sum_{n\ge2}\hbar^n\Phi_{a,n},\qquad
 \Phi_{a,n}\in\operatorname{Diff}_{R[a]}(A_a).
\]
It fixes $R[a,q_a][[\hbar]]$ pointwise.
\end{lemma}
\begin{proof}
Write $\Phi_a=\id+\sum_{n\ge1}\hbar^n\Phi_{a,n}$ and compare
coefficients in
\[
 \Phi_a(f\star_{\mathrm{exp},a}g)
       =\Phi_a(f)\star_{\mathrm K,a}\Phi_a(g).
\]
At order one both products have coefficient $\pi_a/2$, hence
$b\Phi_{a,1}=0$. Thus $\Phi_{a,1}$ is an $R[a]$-derivation of
$A_a$. It vanishes on $x,y,z$, so it is zero.
For $n\ge2$, suppose the coefficients of smaller index are
differential. The order-$n$ equation expresses
$b\Phi_{a,n}$ as a finite sum of compositions and products of
bidifferential product coefficients and the already differential
$\Phi_{a,j}$ with $j<n$. It is therefore bidifferential.
Since $\Phi_{a,n}(1)=0$, Lemma~\ref{lem:sym}, applied over $R[a]$,
shows that $\Phi_{a,n}$ is a relative differential operator.
This proves the assertion by induction. On the polynomial algebra
$A_a$, such an operator is a finite sum of monomial partial
derivatives with coefficients in $R[a,x,y,z]$.

The connection $H$ acts by derivations of both products because
their coefficients are independent of $s,t,u,x$. It also acts on $U_{\hbar,a}$ by the
corresponding Euler derivations on base coefficients and
\[
 \begin{aligned}
 H_s(X)&=1, & H_s(Y_0)&=H_s(Z)=0,\\
 H_r(X)&=H_r(Y_0)=H_r(Z)=0 &&(r=t,u).
 \end{aligned}
\]
These formulas preserve the defining relations. Each $J$ intertwines
the resulting derivations, as can be checked on coefficients and
generators. Hence $\Phi_a H_r=H_r\Phi_a$ for $r=s,t,u$.

To see what happens to the curving, use the affine coordinates
$(x,y,q_a)$. Taking derivatives at fixed $q_a$ gives
\[
                        \pi_a=(q_a+ax)\partial_x\wedge\partial_y.
\]
The graph formula is equivariant under this change by the chain rule
\cite[Sect.~2]{KonFormality}; all derivatives remain relative to $R[a]$.
Every graph edge therefore differentiates in $x$ or $y$.
A positive-order graph with an edge entering $f(q_a)$ vanishes.
Terms with no edge entering that argument factor through $f(q_a)$,
and their sum is zero by unitality. Thus multiplication by $f(q_a)$
is ordinary for the Kontsevich product, as it is for the exponential
product. Since $\Phi_a$ fixes $q_a$ and is multiplicative, it fixes
all its polynomials; continuity gives the assertion for formal series.
\end{proof}

Every coefficient of $\Phi_a$ and its inverse is differential and
polynomial in $a$. Specialisation at $a=\hbar$ therefore involves
only finitely many operators at each order and gives a differential
equivalence
\[
 \Phi:(A[[\hbar]],\star)\xrightarrow{\ \sim\ }
             (A[[\hbar]],\star_{\mathrm K}).
\]
It commutes with $H$ and fixes $q$ and its polynomials, so we can
compare lifts of the same unfolding $(H,C_f)$.

\begin{proposition}[The obstruction for the Kontsevich product]
\label{cor:disk}
For the prescribed product $\star_{\mathrm K}$, the first-order
obstruction of $(H,C_f)$ is
\[
 d_3[\beta]=[-f'(z)]\in\mathbb R[z]/(zf'(z)).
\]
The pair admits a first-order Hamiltonian lift over $A$ precisely
when $f$ is constant.
\end{proposition}
\begin{proof}
Write $\Phi=\id+\hbar^2Q+O(\hbar^3)$. The intertwining identity
at order two gives
$B^{\mathrm K}_2=B_2-bQ$. Since $A$ is commutative, $bQ$ is
symmetric. Thus $\Alt B^{\mathrm K}_2=\Alt B_2=\beta$.
The classical coupling complex is determined by $(H,C_f)$, which
is unchanged. Proposition~\ref{prop:third} therefore gives the
displayed class and the lifting criterion.
\end{proof}

For fixed $H$, Lemma~\ref{lem:kontsevich-differential-comparison}
identifies the two centres with $R[q][[\hbar]]$ by the identity map.
It therefore identifies their cochain complexes and sends the
preliminary curving to $f(q)\omega$ with the same Bianchi defect
$-\hbar f'(q)\nu$. As in Corollary~\ref{cor:heis-anomaly-universal},
these defects realise every central degree-three class.

Finite-order differential operators extend uniquely under
localisation~\cite[Tag~0G36]{Stacks}. Applying this to the
coefficients of $\Phi$ and its inverse extends their equivalence
to $A[z^{-1}]$. There put $D^{\mathrm K}_\xi=\Phi D_\xi\Phi^{-1}$
for $D$ from~\eqref{eq:regularlift} and write
$I_b^{\mathrm K}=[b,-]_{\star_{\mathrm K}}/\hbar$.

\begin{corollary}[A Kontsevich lift on the regular locus]
\label{cor:kontsevich-regular-lift}
The pair $(D^{\mathrm K},f(q)\omega)$ is a full Hamiltonian lift
of $(H,C_f)$ over $A[z^{-1}]$ for every $f\in\mathbb R[z]$.
\end{corollary}
\begin{proof}
Transport the full lift of Corollary~\ref{cor:heis-regular-lift}
by the differential equivalence, as in Section~\ref{sec:lifts}.
Explicitly,
\[
 F_{D^{\mathrm K}}=\Phi F_D\Phi^{-1}=I^{\mathrm K}_{f(q)\omega},
 \qquad d_{D^{\mathrm K}}(f(q)\omega)
                     =\Phi\bigl(d_D(f(q)\omega)\bigr)=0.
\]
Reduction modulo $\hbar$ recovers $(H,C_f)$.

\end{proof}

\begin{remark}[Dependence on the prescribed product]
At $a=0$, the Kontsevich graph operators differentiate only in
$x,y$, while the coefficient $z$ is constant in these directions.
Every graph with an internal edge therefore vanishes. The remaining
terms give the fibrewise Moyal formula~\cite[Section~1.4.1]{KonFormality}
\[
 g\star_{\mathrm K,0}h
 =m\exp\!\left[\frac{\hbar z}{2}
 (\partial_x\otimes\partial_y-\partial_y\otimes\partial_x)\right]
 (g\otimes h),
\]
where each power of $z$ multiplies the result. The connection $H$
preserves this product, and $f(z)\omega$ is central and $H$-closed,
so $(H,C_f)$ lifts for every $f$. Specialising instead at $a=\hbar$
gives the product $\star_{\mathrm K}$ studied above, with the same
classical bracket and the formal correction $\hbar\beta$.
By Proposition~\ref{cor:disk}, a global first-order lift for this
choice requires $f$ to be constant.
\end{remark}

\begin{remark}[Interpretation by polynomial boundary observables]
\label{rem:heis-disk-interpretation}
For a real value of the auxiliary parameter $a$, let $\Sigma$ be an
oriented disk and consider smooth fields
$\phi:\Sigma\to\mathfrak g_a^*$ and
$\eta\in\Omega^1(\Sigma;\mathfrak g_a)$, with the tangential
boundary condition $\iota^*\eta=0$ for
$\iota:\partial\Sigma\hookrightarrow\Sigma$.
The linear Poisson action is
\[
 \mathcal S_a(\phi,\eta)=\int_\Sigma
 \left(\eta_i\wedge d\phi^i+
               \tfrac12\pi_a^{ij}(\phi)\eta_i\wedge\eta_j\right)
 =\int_\Sigma\left\langle\phi,
                         d\eta+\tfrac12[\eta,\eta]_a\right\rangle.
\]
Here repeated coordinate indices are summed, and
$\langle-,-\rangle$ is the pairing between $\mathfrak g_a^*$ and
$\mathfrak g_a$. The second equality follows by integration by
parts. The boundary term
$-\int_{\partial\Sigma}\langle\phi,\eta\rangle$ vanishes under the
stated condition. This is the two-dimensional BF
action for $\mathfrak g_a$~\cite[Section~3.1]{CF}.

Cattaneo--Felder identify the Kontsevich product with the
renormalised perturbative product of two boundary insertions,
expanded around a constant map with its value fixed at a third
boundary point \cite[Sects.~3.1, 3.7 and~3.8]{CF}.
We use their tadpole prescription (Section~3.8), since
$\partial_i\pi_a^{ij}\partial_j=a\partial_y$ in the linear
coordinates, and normalise the leading commutator to $\hbar\pi_a$.
Polynomial insertions give the polynomial subalgebra used above.
The algebraic substitution $a=\hbar$ combines this auxiliary family
into the prescribed quantisation of $\pi+\hbar\beta$.

On polynomial boundary observables with coefficients in $R$, the
given connection $H$ already acts by derivations. The
class~\eqref{eq:fobstruction} tests whether a first-order lift of
$H$ can be equipped with curvature Hamiltonians lifting $C_f$
and satisfying the Bianchi equation. The algebraic dependence on the
parameters enters through the residue calculation. In particular,
$ds/s$ represents a nonzero algebraic de~Rham class on the parameter
torus, although it has the primitive $\log|s|$ on its smooth real locus.

\end{remark}

\appendix

\section{The successive obstruction theorem}\label{app:successive-obstructions}

In Theorem~\ref{thm:tower} the three leading defects are the
components of the first possible nonzero coefficient of a Jacobiator. We first
prove that this gives a coupling cocycle, then examine how changing
the tentative coefficients changes that cocycle.

Fix $n\ge1$ and a partial lift $s_{n-1}$. Choose $D_n$ satisfying
the product equation through order $n$, a two-form $c_n$ and any
higher coefficients of the prescribed kinds. Thus
\[
 \mathcal P=O(\hbar^{n+1}),\qquad \mathcal K=O(\hbar^n),
 \qquad \mathcal J=O(\hbar^n).
\]
In the notation of~\eqref{eq:ob}, put
$E_\xi=[\hbar^{n+1}]\mathcal P_\xi$, so $A_{n,\xi}=\Alt E_\xi$.

\subsection{The leading defect cocycle}

We first verify that the defects depend on the relative differentials
of their function arguments. This will let us read the Jacobiator
as a cochain of $L$.

\begin{lemma}[Symbols of the leading defects]\label{lem:defect-symbols}
Each $E_\xi$ is a normalised relative differential Hochschild
two-cocycle. Its alternating part $A_{n,\xi}$ is a biderivation,
and $K_{n,\xi\eta}$ is a relative derivation. Consequently
$O_n=(A_n,K_n,J_n)$ defines a cochain in $(F^1)^3$.
\end{lemma}
\begin{proof}
The connection rule on $R$ makes $\mathcal P_\xi$ $R$-bilinear.
Applying $D_\xi$ to associativity gives $b_\star\mathcal P_\xi=0$,
where $b_\star$ is the Hochschild differential for $\star$. Since
$\mathcal P_\xi=O(\hbar^{n+1})$, the leading coefficient gives
$bE_\xi=0$ for ordinary multiplication. Its normalisation follows
from $D_\xi(1)=0$ and the unit of $\star$.
As in the coefficient computation following Lemma~\ref{lem:sym},
the contribution of the projectable field $H_\xi$ is the relative
bidifferential cochain $\cL_{H_\xi}B_{n+1}$. All remaining terms
are compositions of relative differential operators.

Consider the associative first-order product
\[
 a\mathbin{\circ} b=ab+\epsilon E_\xi(a,b),\qquad \epsilon^2=0.
\]
Its commutator is $\epsilon A_{n,\xi}(a,b)$. The identity
$[a,b\circ c]_\circ=[a,b]_\circ\circ c+b\circ[a,c]_\circ$
at the coefficient of $\epsilon$ gives
\[
 A_{n,\xi}(a,bc)=A_{n,\xi}(a,b)c+bA_{n,\xi}(a,c).
\]
Alternation gives the same Leibniz rule in the other argument.
Thus $A_{n,\xi}$ is a relative biderivation.

Modulo $\hbar^{n+1}$, each $D_\xi$ is a derivation of $\star$.
Commutators of derivations and the term $D_{[\xi,\eta]}$ therefore
make $F_D$ a derivation to that order. The inner operator
$I_{C_{\hbar,\xi\eta}}$ is a derivation at every order, so
\[
 \mathcal{K}_{\xi\eta}(a\star b)
 -\mathcal{K}_{\xi\eta}(a)\star b-a\star \mathcal{K}_{\xi\eta}(b)
             =O(\hbar^{n+1}).
\]
Its coefficient at $\hbar^n$, using $\mathcal{K}=O(\hbar^n)$, is the ordinary
Leibniz rule for $K_{n,\xi\eta}$. Moreover, curvature and inner
operators both vanish on $R$, so this derivation is relative.

The dependence on $\xi,\eta,\zeta$ is alternating and $R$-multilinear
by the connection law. The universal property of $\Omega^1_{A/R}$
therefore gives the required $A$-multilinear cochain through
\[
\begin{aligned}
 O_n(\widehat\xi,\dd a,\dd b)&=A_{n,\xi}(a,b),\\
 O_n(\widehat\xi,\widehat\eta,\dd a)&=K_{n,\xi\eta}(a),\\
 O_n(\widehat\xi,\widehat\eta,\widehat\zeta)&=J_{n,\xi\eta\zeta},\\
 O_n(\dd a,\dd b,\dd c)&=0.
\end{aligned}
\]
The last equality places $O_n$ in $F^1$.
\end{proof}

On $E_\hbar=(A\oplus T)[[\hbar]]$ use the bracket of
Remark~\ref{rem:quantum-extension} with the tentative coefficients.
Writing $a=(a,0)$ and $e_\xi=(0,\xi)$, its values on generators are
\begin{equation}\label{eq:almost}
 [a,b]_\hbar=\mu(a,b),\quad [e_\xi,a]_\hbar=D_\xi a,\quad
 [e_\xi,e_\eta]_\hbar=e_{[\xi,\eta]}+C_{\hbar,\xi\eta}.
\end{equation}
Its classical reduction $E_0=A\oplus T$ acts on $A$ by
$(a,\xi)\cdot f=X_a(f)+H_\xi(f)$.

\begin{lemma}[The leading Jacobiator]\label{lem:defect-cocycle}
The cochain $O_n$ is closed for the coupling differential:
$d_LO_n=0$.
\end{lemma}
\begin{proof}
Let
\[
 \operatorname{Jac}_\hbar(u,v,w)
       =[u,[v,w]_\hbar]_\hbar+[v,[w,u]_\hbar]_\hbar
                            +[w,[u,v]_\hbar]_\hbar.
\]
On three vertical entries this expression is zero by Jacobi for
$\mu$. On one horizontal and two vertical entries, expansion of
\eqref{eq:almost} gives
\[
 \operatorname{Jac}_\hbar(e_\xi,a,b)
   =D_\xi\mu(a,b)-\mu(D_\xi a,b)-\mu(a,D_\xi b)
   =\hbar^{-1}\Alt \mathcal{P}_\xi(a,b).
\]
For two horizontal entries, the same expansion gives
\[
 \operatorname{Jac}_\hbar(e_\xi,e_\eta,a)
   =\bigl([D_\xi,D_\eta]-D_{[\xi,\eta]}
                              -I_{C_{\hbar,\xi\eta}}\bigr)(a)
   =\mathcal{K}_{\xi\eta}(a).
\]
For three horizontal entries, Jacobi in $T$ cancels the $T$-component
and the $A$-component is
\[
 \operatorname{Jac}_\hbar(e_\xi,e_\eta,e_\zeta)
            =(d_D C_\hbar)_{\xi\eta\zeta}=\mathcal{J}_{\xi\eta\zeta}.
\]
Thus $\operatorname{Jac}_\hbar=O(\hbar^n)$, with values in the
vertical summand $A[[\hbar]]$. The tentative bracket satisfies the
Leibniz rule over $R[[\hbar]]$ with anchor $(a,\xi)\mapsto\xi$,
and this anchor preserves brackets. These two identities make its
Jacobiator $R[[\hbar]]$-multilinear.

The Nijenhuis--Richardson bracket is the graded commutator of
insertions of alternating multilinear maps, with an $r$-linear
map placed in degree $r-1$~\cite{NR}. Its graded Jacobi identity
gives $[B,[B,B]]=0$ for an alternating bilinear map $B$.
Apply this identity to~\eqref{eq:almost}. Since all earlier
coefficients of its Jacobiator vanish, the coefficient at $\hbar^n$
is the Chevalley--Eilenberg identity
\[
 d_{E_0}\bigl([\hbar^n]\operatorname{Jac}_\hbar\bigr)=0
\]
for the action of $E_0$ on $A$ displayed above. Concretely, at this
coefficient the outer bracket is the classical bracket, and its
action on the vertical value of the Jacobiator is $X_a+H_\xi$.

Recall the comparison map $\chi:E_0\to L$ from the construction
of the coupling bracket, $\chi(a,\xi)=\dd a+\widehat\xi$.
The generator formulas give
$\chi([u,v]_{E_0})=[\chi(u),\chi(v)]_L$ and $a_L(\chi(u))f=u\cdot f$.
Pullback of cochains therefore gives
\[
\begin{CD}
 C^q(L;A) @>{d_L}>> C^{q+1}(L;A)\\
 @V{\chi^*}VV @VV{\chi^*}V\\
 C_R^q(E_0;A) @>{d_{E_0}}>> C_R^{q+1}(E_0;A),
\end{CD}
\]
where the bottom row consists of alternating $R$-multilinear
cochains. The image of $\chi$ generates $L$ over $A$, so
$A$-multilinearity makes the vertical arrows injective. The
Jacobiator computations above identify
$[\hbar^n]\operatorname{Jac}_\hbar=\chi^*O_n$. Thus
$\chi^*(d_LO_n)=d_{E_0}(\chi^*O_n)=0$, which proves closedness.
\end{proof}

\subsection{Corrections and extension classes}

We now vary the tentative coefficients. Their effect on $O_n$
identifies both the obstruction to extension and the freedom when
an extension exists.

\begin{lemma}[Variation of the defect cocycle]\label{lem:defect-variation}
For two tentative extensions of $s_{n-1}$, write their order-$n$
coefficient differences as $(W,c)$. Then
$W\in\Omega_T^1(\Der_R A)$ and
\[
 O'_n-O_n=d_L(W,c).
\]
Consequently $[O_n]\in H^3(F^1)$ depends only on $s_{n-1}$ and
the prescribed product.
\end{lemma}
\begin{proof}
Both admissible auxiliary coefficients $D_n$ solve the order-$n$
derivation equation for the same lower coefficients. Subtracting those
equations gives $bW=0$, so each $W_\xi$ is a relative derivation.
Conversely, adding any such $W$ preserves that equation.

The change $D\mapsto D+\hbar^nW$, $C_\hbar\mapsto
C_\hbar+\hbar^nc$ changes the three leading defects by
\[
\begin{aligned}
 \Delta A_n&=\cL_W\pi,\\
 \Delta K_n&=d_HW-X_c,\\
 \Delta J_n&=d_Hc+W\cdot C.
\end{aligned}
\]
For the first equality, the contribution to $\mathcal{P}$ at order $n+1$ is
$\cL_WB_1$, whose alternating part is $\cL_W\pi$.
For the second, the linear terms in the curvature are
$[H_\xi,W_\eta]-[H_\eta,W_\xi]-W_{[\xi,\eta]}$, and the
new curving contributes $-X_c$.
For the third, vary $d_D$ on $C$ and the curving on which it acts.
Products of two new coefficients have order at least $2n\ge n+1$,
so~\eqref{eq:FilDiff} identifies the displayed variation with $d_L(W,c)$.

Changing $D_{n+1}$ by $U$ changes $E_\xi$ by $-bU_\xi$.
For a commutative product this coboundary is symmetric, so its
alternating part is zero. The coefficient $K_n$ depends on connection
and curving coefficients through order $n$, as does $J_n$.
Thus the triple is unchanged by the remaining auxiliary choices.
Together with Lemma~\ref{lem:defect-cocycle}, the variation formula
proves the assertion about its cohomology class.
\end{proof}

\begin{lemma}[Completion of a tentative extension]\label{lem:defect-completion}
A tentative pair $(D_n,c_n)$ extends $s_{n-1}$ through order $n$
if and only if its leading defect cochain $O_n$ vanishes.
\end{lemma}
\begin{proof}
Suppose $O_n=0$. Then $\mathcal{K},\mathcal{J}=O(\hbar^{n+1})$ and each
$E_\xi=[\hbar^{n+1}]\mathcal{P}_\xi$ is symmetric. By
Lemma~\ref{lem:defect-symbols} it is a normalised differential
Hochschild cocycle. Lemma~\ref{lem:sym} therefore supplies a
normalised relative differential operator $U_\xi$ with
$bU_\xi=E_\xi$.

As in Lemma~\ref{lem:simultaneous-corrections}, projectivity of
$T$ gives an $R$-linear choice of primitives and finite generation
gives a common bound on their orders.

Replace $D$ by $D+\hbar^{n+1}U$. At order $n+1$ its Leibniz
defect changes by $-bU=-E$, so now $\mathcal{P}=O(\hbar^{n+2})$.
The curvature and Bianchi equations remain valid modulo
$\hbar^{n+1}$. This is exactly the partial-lift condition through
order $n$.

Conversely, an extension satisfies the three defect bounds of
Definition~\ref{def:partial-lift} for some auxiliary $D_{n+1}$.
Thus $O_n=0$ for that auxiliary choice and, by
Lemma~\ref{lem:defect-variation}, for every choice.
\end{proof}

\begin{lemma}[Transport of the obstruction]\label{lem:defect-naturality}
A formal splitting change reducing to $b$ transports the obstruction
by pullback along the classical isomorphism of
Lemma~\ref{lem:coupling-bracket}. A relative differential product
equivalence preserves the obstruction in the same coupling complex.
\end{lemma}
\begin{proof}
Let $\lambda\in\Omega_T^1(A_\hbar)$ be the splitting change, and write
primes for the transported data. Formula~\eqref{eq:qchange} gives a
bracket-preserving map from the new tentative extension to the old one,
\[
 S_\lambda(a,\xi)=(a+\lambda_\xi,\xi).
\]
It fixes the vertical summand, which contains all Jacobiator values.
Adding the inner derivations $I_{\lambda_\xi}$ preserves the product
Leibniz defect, and the transported Jacobiator is still $O(\hbar^n)$.
Thus the transported coefficients remain admissible through order $n-1$.
Writing $S_b$ for the classical reduction, the leading Jacobiators
satisfy $\chi'^*O'_n=S_b^*\chi^*O_n$. On coupling cochains
the classical change of splitting $\psi_b$ of
Lemma~\ref{lem:coupling-bracket} fits in the square
\[
\begin{CD}
 E'_0 @>{S_b}>> E_0\\
 @V{\chi'}VV @VV{\chi}V\\
 L' @>{\psi_b}>> L .
\end{CD}
\]
Both routes send $(a,\xi)$ to $\dd(a+b_\xi)+\widehat\xi$.
Since $\psi_b$ fixes vertical generators, its pullback preserves $F^1$.
The square gives $\chi'^*O'_n=\chi'^*\psi_b^*O_n$, so
injectivity of $\chi'^*$ yields $O'_n=\psi_b^*O_n$.
For $b=0$ the coupling map is the identity.

A product equivalence $\Phi$ intertwines the tentative brackets by
$(a,\xi)\mapsto(\Phi(a),\xi)$. It also conjugates the product
Leibniz defect. Since $\Phi\equiv\id\pmod\hbar$, the leading
Jacobiator and the classical comparison map are unchanged. This
gives $O'_n=O_n$ for the transported choices. For either transport
Lemma~\ref{lem:defect-variation} shows that further choices yield
the same cohomology class.
\end{proof}

\begin{proof}[Proof of Theorem~\ref{thm:tower}]
Lemmas~\ref{lem:defect-cocycle} and~\ref{lem:defect-variation}
define the class $o_n(s_{n-1})=[O_n]$. If it vanishes, write
$O_n=d_L(W,c)$ and replace $(D_n,c_n)$ by
$(D_n-W,c_n-c)$. Lemma~\ref{lem:defect-variation} makes the new
leading cocycle zero, and Lemma~\ref{lem:defect-completion} completes
the extension. Conversely, an extension has zero leading defect
and hence zero obstruction class.

Fix one extension. By the same two lemmas, the difference between
it and any other extension is exactly a cocycle
$(W,c)\in Z^2(F^1)$, and every such cocycle gives another extension.
A splitting change $\hbar^n\alpha$ acts on the order-$n$ pair by
\[
 (D_n,c_n)\longmapsto
 (D_n+X_\alpha,\ c_n+d_H\alpha).
\]
This is addition of $d_L\alpha$. Since $(F^1)^1=\Omega_T^1(A)$,
all degree-two boundaries arise in this way. The equivalence classes
therefore form a torsor for
$Z^2(F^1)/d_L(F^1)^1=H^2(F^1)$.
Lemma~\ref{lem:defect-naturality} proves the final assertion.
\end{proof}

\par\medskip
\noindent\textit{Declaration on use of AI.}\enspace
During the preparation of this manuscript, the authors benefited from adversarial critiques provided by ChatGPT. This process informed the present organization of the paper, particularly the development of the Appendix and the proof of Theorem 4.14. The authors take full responsibility for the final mathematical content of the paper.

\begingroup
\renewcommand{\bibliofont}{\small}

\endgroup
\end{document}